\documentclass[11pt]{amsart}

\title{Simplicial cohomology of band semigroup algebras}
\author{Y. Choi, F. Gourdeau and M. C. White}
\date{12th April 2010, some corrections 27th June 2011}

\usepackage{amsmath,amssymb,amsthm}

\RequirePackage[usenames]{color}  
\newcommand{\YCedit}[1]{}
\newcommand{\FGedit}[1]{}

\newcommand{\para}[1]{\paragraph{\bf#1.}}
\newcommand{\dt}[1]{\textcolor{Maroon}{\sf #1}}   
\newcommand{\st}{\;:\;}
\newcommand{\defeq}{:=}

\newcommand{\abs}[1]{\vert#1\vert}
\newcommand{\norm}[1]{\left\| #1 \right\|}

\newcommand{\id}{{\sf 1}}  
\newcommand{\ran}{\operatorname{ran}}  
\newcommand{\lto}[1]{{\stackrel{#1}{\longleftarrow}}}

\newcommand{\Cplx}{{\mathbf C}}   

\newcommand{\al}{\alpha}
\newcommand{\lm}{\lambda}

\renewcommand{\aa}{{\bf a}}
\newcommand{\bb}{{\bf b}}
\newcommand{\yy}{{\bf y}}

\newcommand{\cA}{{\mathcal A}} 
\newcommand{\cF}{{\mathcal F}} 

\newcommand{\bbG}{{\mathbb G}}

\renewcommand{\mod}{\operatorname{mod}} 

\newcommand{\tp}{\otimes}
\newcommand{\ptp}{\widehat{\otimes}}  
\newcommand{\filler}[1]{\mathop{\bullet}_{#1}} 

\newcommand{\dif}{\delta}     
\newcommand{\face}{\partial}  
\newcommand{\cyc}{\operatorname{\bf t}} 
  
  \newcommand{\Ho}[3][]{{\mathcal {#2}^{#1}_{#3}}}

\newcommand{\HH}{{\mathcal H\mathcal H}} 
\newcommand{\CC}{\mathcal C\mathcal C} 
\newcommand{\HC}{\mathcal H\mathcal C} 
\newcommand{\BC}{\mathcal B\mathcal C} 
\newcommand{\ZC}{\mathcal Z\mathcal C} 

\newcommand{\SB}[2]{{\mathcal B}^{#1}({#2},{#2}')} 
\newcommand{\SC}[2]{{\mathcal C}^{#1}({#2},{#2}')} 
\newcommand{\SZ}[2]{{\mathcal Z}^{#1}({#2},{#2}')} 

\renewcommand{\d}{\delta}
\newcommand{\iso}{\cong}

\newcommand{\base}[1]{\operatorname{#1}}

\newcommand{\CR}{\base{CR}}
\newcommand{\ZR}{\base{ZR}}
\newcommand{\BR}{\base{BR}}

\newcommand{\lcu}[1]{\langle#1]}  
\newcommand{\high}{\operatorname{ht}}  
\newcommand{\sR}{{\mathsf R}}
\newcommand{\err}{\operatorname{Err}}
\newcommand{\desc}{\operatorname{desc}}
\newcommand{\lin}{\operatorname{lin}}

\newcommand{\SD}{\operatorname{SD}}
\newcommand{\DS}{\operatorname{DS}}

\newcommand{\et}{{\bf x}}  
\newcommand{\chain}{x_1 \otimes \cdots \otimes x_{n+1}}

\newcommand{\subchain}[2]{x_{#1} \otimes \cdots \otimes x_{#2}}
\newcommand{\achain}{a_1 \otimes \cdots \otimes a_{n+1}}
\newcommand{\asubchain}[2]{a_{#1} \otimes \cdots \otimes a_{#2}}
\newcommand{\lbchain}{w_1 \otimes \cdots \otimes w_j}
\newcommand{\lbsubchain}[2]{w_{#1} \otimes \cdots \otimes w_{#2}}

\newcounter{pulse}
\numberwithin{pulse}{section}
\numberwithin{equation}{section}

\newtheorem{Proposition}[pulse]{Proposition}
  \newtheorem{prop}[pulse]{Proposition}
\newtheorem{Theorem}[pulse]{Theorem}
  \newtheorem{thm}[pulse]{Theorem}
\newtheorem{Lemma}[pulse]{Lemma}
  \newtheorem{lem}[pulse]{Lemma}

  \newtheorem{cor}[pulse]{Corollary}
\theoremstyle{definition}
  \newtheorem{Definition}[pulse]{Definition}
  \newtheorem{dfn}[pulse]{Definition}
  \newtheorem{notn}[pulse]{Notation}
  \newtheorem{Example}[pulse]{Example}
\theoremstyle{remark}
  \newtheorem{Remark}[pulse]{Remark}
  \newtheorem{rem}[pulse]{Remark}

\newenvironment{YCnum}{%
\begin{enumerate}
\renewcommand{\labelenumi}{{\rm(\roman{enumi})}}
\renewcommand{\theenumi}{{\rm(\roman{enumi})}}
}{\end{enumerate}\ignorespacesafterend}

\begin{document}
\begin{abstract}
We establish simplicial triviality of the convolution algebra $\ell^1(S)$, where $S$ is a band semigroup. This generalizes results of the first author~\cite{Ch1,YC_SA}.
To do so, we show that the cyclic cohomology of this algebra vanishes in all odd degrees, and is isomorphic in even degrees to the space of continuous traces on $\ell^1(S)$.
Crucial to our approach is the use of the structure semilattice of~$S$, and the associated grading of $S$, together with an inductive normalization procedure in cyclic cohomology; the latter technique appears to be new, and its underlying strategy may be applicable to other convolution algebras of interest.

\end{abstract}
\maketitle

\begin{section}{Introduction}
Computing Hochschild cohomology of Banach algebras has remained a difficult task, even when restricted to the class of $\ell^1$-convolution algebras of semigroups: see \cite{BoDu,DaDu} for earlier work on various examples, albeit only in low dimensions.
Previous work of the first author~\cite{Ch1,YC_SA} showed that the simplicial cohomology of the semigroup algebra $\ell^1(S)$ vanishes when $S$ is a \emph{normal band}; however, the techniques were unable to handle the case of general band semigroups.
(We note that bands comprise a rich and interesting class of semigroups: particular kinds of band have been studied both in abstract semigroup theory, and also in operator-theoretic settings~\cite{LMMR_JOT98, LMMR_JOT01}.)

In this paper we calculate all the cyclic and simplicial cohomology groups of $\ell^1(S)$ where $S$ is an \emph{arbitrary} band semigroup. More precisely, we shall show the following:
\begin{itemize}
\item[--] the cyclic cohomology of $\ell^1(S)$ is isomorphic in even degrees to the space of continuous traces on $\ell^1(S)$, and vanishes in odd degrees (Theorem~\ref{t:headline});
\item[--] the simplicial cohomology of $\ell^1(S)$ vanishes in all strictly positive degrees (Theorem~\ref{t:headline2}).
\end{itemize}
The techniques used in establishing these results resemble those in earlier work of the second and third authors \cite{GJW}, in that one performs explicit calculations with cyclic cochains, and then uses the Connes-Tzygan long exact sequence to calculate the simplicial cohomology. As in that paper, the decision to work with cyclic cohomology really is forced upon us by the nature of our construction (see Corollary~\ref{c:now-we-need-cyclic} below), and is not merely incidental.
%

Some of our results appear to generalize to the setting of Banach algebras which are $\ell^1$-graded over a semilattice.
In particular, it seems that similar calculations would provide an alternative approach to some of the first author's existing results for Clifford semigroups in~\cite{YC_SA}.
However, we shall focus throughout on the case of band semigroup algebras, to keep the exposition reasonably self-contained.

One approach which one might be tempted to adopt, in order to prove that band
semigroup algebras have trivial cyclic cohomology, is to exhaust the
band by finitely generated bands and cobound the cocycle on increasingly
large sets.

This is even more tempting when one recalls that finitely generated
bands are finite,~\cite[Theorem IV.4.9]{Ho} (or see \cite{BL_finbands} for a short, direct proof).
However, one encounters problems with this approach.
It is difficult to obtain uniform control of the norms of the coboundaries
as we take larger and larger generating sets for these bands.
This is true even in the commutative case, which corresponds to
the setting of~\cite{Ch1}.
Another feature is that finite band algebras are, in general, neither semisimple nor amenable, which makes their trivial simplicial
cohomology surprising.

It should nevertheless be noted that, by specializing the present arguments to the case of a \emph{semilattice}~$L$, one obtains a direct calculation of the cyclic cohomology of $\ell^1(L)$. Previously, this was only known by applying the Connes-Tzygan exact sequence and using the main result of \cite{Ch1}.
Moreover, in order to apply the Connes-Tzygan exact sequence, one first has to show that certain obstruction groups vanish -- and the only previous proof that these obstructions vanish relied indirectly on other results from~\cite{Ch1}. Thus the methods of the present paper give a much more accessible proof that $\ell^1(L)$ has the same cyclic cohomology as the ground field.

\begin{rem}
A feature which may be of wider interest is that, rather than constructing a splitting homotopy directly on the cyclic cochain complex, we construct maps which split ``modulo terms of lower order'' in a particular filtration,
and then employ an iterative procedure to move progressively further down the filtration.
Some of these arguments could be cast in terms of a more general theory of cohomology of filtered complexes; however, this seems to bring little extra advantage or clarity for the present problem, and so we shall carry out our iterative reduction in a hands-on fashion.
\end{rem}
\end{section}

\begin{section}{Notation and preliminaries}

\subsection{Cohomology}
Since this paper is only concerned with simplicial and cyclic cohomology, rather than Hochschild cohomology with more general coefficients, we shall present a fairly minimal set of definitions that is enough for our purposes.
Our terminology is that of~\cite{GJW}, but with some small differences of notation.

Let $\cA$ be a Banach algebra and regard $\cA'$\/, the dual space of $\cA$, as a Banach $\cA$-bimodule in the usual way.
As in \cite[\S1]{GJW}, for $n\geq 0$, $\SC{n}{\cA}$ denotes the space of \dt{$n$-cochains}, $\SZ{n}{\cA}$ the subspace of \dt{$n$-cocycles}, and $\SB{n}{\cA}\subseteq\SZ{n}{\cA}$ the subspace of \dt{$n$-coboundaries}.
Note that by convention, $\SC{0}{\cA}=\cA'$ and
$\SC{n}{\cA}=0$ for negative $n$. Our notation for the corresponding cohomology groups differs from that of~\cite{GJW}: we shall write
$\HH^n(\cA)$ for the quotient space $\SZ{n}{\cA}/\SB{n}{\cA}$\/. This is the $n$th \dt{simplicial cohomology group} of~$\cA$.

We need to specify some notation for the \dt{Hochschild coboundary operator} $\dif^n:\SC{n}{\cA}\to\SC{n+1}{\cA}$. Recall (cf.~\cite{GJW}) that an $n$-cochain is a bounded $n$-linear map $T:\cA^n\to \cA'$, and that the $(n+1)$-cochain $\dif^nT$ is defined by
\begin{eqnarray*}
(\d^n T)(a_1,\ldots,a_{n+1})(a_{n+2}) &=& T(a_2, a_3,  \ldots, a_{n+1})(a_{n+2}a_1) \\
& & + \sum_{j=1}^{n} (-1)^j T(a_1, a_2, \ldots, a_j a_{j+1}, \ldots, a_{n+1})(a_{n+2}) \\
& & { } + (-1)^{n+1}  T(a_1, \ldots, a_n)(a_{n+1}a_{n+2})
 \end{eqnarray*}
  where $a_1,\ldots,a_{n+2} \in \cA$.
We shall usually omit the superscript and write $\dif$ for $\dif^n$.

For each $n$, elements of $\SC{n}{\cA}$ may be regarded as bounded linear functionals on the space $\Ho{C}{n}(\cA)\defeq A^{\ptp n+1}$, the $n+1$-fold completed projective tensor product of $\cA$\/; if we do this, then the coboundary operator $\delta:\SC{n}{\cA}\to\SC{n+1}{\cA}$ is clearly the adjoint of the operator $d:\Ho{C}{n+1}(\cA)\to\Ho{C}{n}(\cA)$ given by
 \begin{eqnarray*}
d (a_1\tp \cdots\tp a_{n+2}) & = &
 a_2\tp \cdots\tp a_{n+1} \tp a_{n+2}a_1 \\
& & + \sum_{j=1}^{n+1} (-1)^j a_1\tp \cdots\tp a_ja_{j+1}\tp \cdots\tp a_{n+2}
 \end{eqnarray*}
for $a_1,\dots,a_{n+2}\in\cA$\/.
This point of view will be more convenient when $\cA=\ell^1(S)$ for a semigroup~$S$\/. For, since there is a well-known isometric isomorphism of Banach spaces
\[ \ell^1(I)\ptp\ell^1(J)\iso \ell^1(I\times J) \quad\text{for any index sets $I$ and $J$\/,} \]
we shall in what follows identify $\ell^1(S)^{\ptp n}$ with $\ell^1(S^n)$.

Simplicial cohomology is closely linked
to \dt{cyclic cohomology}, which we now introduce.
Denote by $\cyc$ the \dt{signed cyclic shift operator} on the simplicial
chain complex:
\begin{equation}\label{eq:dfn-cyc-shift}
\cyc(\achain) = (-1)^n (a_{n+1}\otimes\asubchain{1}{n}).
\end{equation}
By abuse of notation, we also write $\cyc$ for the adjoint
operator on the simplicial \emph{cochain} complex.
The $n$-cochain $T$ (in $\SC{n}{\cA}$) is called \dt{cyclic} if
$\cyc T = T$ and the linear space of all cyclic $n$-cochains is denoted by~$\CC^n(\cA)$.

It is well known that the cyclic cochains $\CC^n(\cA)$ form a subcomplex of
$\SC{n}{\cA}$, that is $\d\left(\CC^n(\cA)\right) \subseteq \CC^{n+1}(\cA)$, and
this allows one to define cyclic versions of the spaces defined above, denoted here by $\ZC^n(\cA)$, $\BC^n (\cA)$ and $\HC^n (\cA)$.
Under certain conditions on the algebra $\cA$ (see \cite{He1}), the cyclic and simplicial cohomology groups are connected via the
Connes-Tzygan long exact sequence
\begin{equation}\label{eq:CTseq}
\cdots\to \HH^n(\cA) \xrightarrow{B} \HC^{n-1}(\cA) \xrightarrow{S} \HC^{n+1}(\cA) \xrightarrow{I} \HH^{n+1}(\cA) \to \cdots
\end{equation}
where the maps $B$, $S$ and $I$ all behave naturally with respect to algebra homomorphisms. (Although we use $S$ to denote both the shift map in cyclic cohomology and a band semigroup, this should not lead to any confusion.)
The reader is referred to~\cite{He1} for more details.

\medskip 
We now introduce some definitions and notation which will be useful in our work.

\begin{dfn}[Cyclic cocycles arising from traces]
Let $A$ be a Banach algebra. Given $\psi\in A'$ and $n\geq 0$\/, let $\psi^{(n)}\in \SC{n}{A}$ be the cochain defined by
\[ \psi^{(n)}(a_1,\ldots, a_n)(a_{n+1}) \defeq \psi(a_1\dotsb a_{n+1})\,. \]
\end{dfn}

\begin{lem}
If $\tau$ is a continuous trace on $A$\/, then $\tau^{(2n)}$ is a cyclic cocycle.
\end{lem}
This is easily verified by a direct calculation, and we omit the proof. 

\begin{Definition}\label{dfn:cyc-equiv}
Two chains $x, y \in \Ho{C}{n}(\cA)$ are \dt{cyclically equivalent} if
\[ x - y \in (I-\cyc)\Ho{C}{n}(\cA).\]
\end{Definition}

\begin{notn}\label{dfn:d_c}
Let $\et=\chain\in \Ho{C}{n}(\cA)$ be an elementary tensor, and suppose we group terms in the tensor together as $\et=\lbchain$, where $j\ge 2$.
We then denote by $d_c(w_l)$ the restriction of $d$ to
$w_l$ when seen as a part of $\et$, meaning that
\[ d_c(w_l)=\sum_{i=\beta_k}^{\beta_k +\alpha_k-2}
(-1)^i x_{\beta_k} \otimes \cdots \otimes x_i \cdot x_{i+1}\otimes
\cdots \otimes x_{\beta_k +\alpha_k -1}\,. \]
 where
$\alpha_l$ is the length of the subtensor~$w_l$\/, and $\beta_l$ the rank of its first element. 
If $w_l$ has length one, i.e.~$\alpha_l=1$, then we define $d_c(w_l)$ to be zero.
\end{notn}

Note that the introduction of $d_c$ is a notational device and does
not define a map on subtensors as the signs are tributary to the
position of this subtensor in the tensor. With this notation, we can
write
\begin{equation}\label{eq:dw1}
d(\lbchain) =
 \left\{ \begin{aligned}
     & \subchain {2}{\alpha_1} \otimes\lbsubchain 2{j}\cdot x_1\\
   + &\sum_{k=1}^{j-1} (-1)^{\beta_{k+1}-1} \lbsubchain 1{k} \cdot \lbsubchain {k+1}{j}\\
   + &\sum_{k=1}^{j}  w_1 \otimes \cdots \otimes d_c(w_k) \otimes \cdots \otimes w_j.
\end{aligned}\right.
\end{equation}

\begin{notn}
In later sections, many of the calculations involve elementary tensors in $\cA^{\ptp n+1}$, of the form $x_1\tp \cdots\tp x_{n+1}$, and their images under certain maps, which have a form like
\begin{equation}\label{eq:longhand}
\subchain{1}{i-1} \tp f(x_i,x_{i+1})\tp g(x_i,x_{i+1})\tp \subchain{i+2}{n+1}
\end{equation}
for certain functions $f,g : \cA\times\cA\to \cA$.
As a notational shorthand, we will often denote such an expression in the abbreviated form
\[ \filler{i-1} \tp  f(x_i,x_{i+1})\tp g(x_i,x_{i+1}) \tp \filler{n-i} \,.\]
\end{notn}

\subsection{Band semigroups}

\begin{Definition}
A semigroup $S$ formed only of idempotents is a \dt{band semigroup}.
\end{Definition}

\FGedit{Definition of normal band needed here or in intro as normal bands are mentioned in the intro. We wouldn't add other examples.}

Of particular importance are rectangular bands.
\begin{Definition}
A \dt{rectangular band} is a semigroup in which the identity $a=aba$ always holds.
\end{Definition}
Note that in any rectangular band, the identity
\[ abc = (aca)bc = a(cabc) = ac \]
holds for arbitrary elements $a$, $b$ and~$c$.
This is particularly clear if one takes the
following description of rectangular bands \cite[Theorem~1.1.3]{Ho}.

\begin{Theorem} Let $S$ be a semigroup. The the following conditions
are equivalent:
\begin{YCnum}
\item $S$ is a rectangular band;
\item $S$ is isomorphic to a semigroup of the form $A \times B$,
where $A$ and $B$ are non-empty sets, and where the multiplication
is given by
\[ (a_1, b_1)(a_2, b_2)= (a_1, b_2)  \qquad\text{ where $a_1,a_2\in A$ and $b_1,b_2\in B$\/.} \]
\end{YCnum}
\end{Theorem}

A commutative band semigroup is called a \dt{semilattice}, and carries a natural and useful partial order defined by $\alpha\preceq\beta \Leftrightarrow \alpha\beta=\alpha$. Semilattices are important in the study of general bands, because of the following structure theorem.

\begin{Theorem}[{\cite[Theorem 4.4.5]{Ho}}]
Any band semigroup $S$ can be represented as a disjoint union
$\coprod_{\alpha\in L} R_\alpha$, where $L$ is a semilattice, each
$R_\alpha$ is a rectangular band given by $A_\alpha \times
B_\alpha$, the left and right index sets, and the following
properties are satisfied:
\begin{YCnum}
  \item $R_\alpha R_\beta\subseteq R_{\alpha\beta}$ for all
  $\alpha,\beta\in L$\/.
  \item For $x=(a_1, b_1) \in R_\alpha$ and $y = (a_2, b_2) \in R_\beta$
 with $\alpha \preceq \beta$, then $xy$ and $y$
  have the same right index (i.e.~$xy = ( \cdot , b_2)$) while $yx$ and $y$ have the same left index (i.e.~$yx = (a_2, \cdot)$).
  \item\label{item:assoc} The product is associative.
\end{YCnum}
\end{Theorem}

Note that condition~\ref{item:assoc} is
needed to ensure that such a construction gives a band semigroup.

\begin{Example}[Normal bands]
A band $S$ is said to be a \dt{normal band} if $xaby=xbay$ for all $a,b,x,y\in S$. In this case, the structure theorem can be sharpened significantly; not only do we get a decomposition $S=\coprod_{\al\in L} R_\al$ into rectangular bands, but this decomposition turns out to exhibit $S$ as a ``strong semilattice of rectangular bands'': see \cite[Prop.~4.6.14]{Ho} for the proof and relevant definitions. In \cite{YC_SA} this stronger decomposition theorem was used to calculate the simplicial cohomology of $\ell^1(S)$; in the present, more general case, new techniques are needed.
\end{Example}

\para{Left coherent units}
Although bands do not in general have units, we will
define \dt{left coherent units} for each element, as follows. Given $S$, for each
rectangular band $R_\alpha$, fix an element $y_\alpha\in R_\alpha$ and
define $\lcu{x} = xy_\alpha$ for each $x\in R_\alpha$. Then the function $\lcu{\,\cdot\,}:S\to S$
has the following properties:\label{temporary-marker}
\begin{itemize}
\item[--] for each $\al\in L$\/, $\lcu{R_\al}\subseteq R_\al$\/;
\item[--] for each $\al\in L$ and each $x\in R_\al$\/, $\lcu{x}x=x$\/;
\item[--] for each $\al,\beta\in L$ such that
$\al\preceq\beta$\/, and each $x\in R_\alpha$ and $y\in R_\beta$\/,
$\lcu{xy}=\lcu{x}$\/.
\end{itemize}

\begin{notn}
As is often done, if $x\in S$ we denote the point mass
at~$x$, as an element of $\ell^1(S)$, by $x$ itself. This should not cause
confusion and follows the notation used in general. Throughout, we
write $\et=\chain$ for an elementary tensor in which each $x_i$
is a point mass at $x_i\in S$.
%
Henceforth, when we speak of an elementary tensor $\et\in \Ho{C}{n}(\cA)$, we shall always mean one of this restricted form.
\end{notn}

To show that one can cobound any $\phi\in \ZC^n(\cA)$,
 it suffices to show that one can do it on these elementary tensors
 (with a uniform bound).
Our strategy will be to proceed by steps, expanding the set of
elementary tensors on which one can cobound at each step.

To do so,
at each step one defines $s: \CC_{n-1}(\cA) \to \CC_n(\cA)$ in such a
way that: denoting by $E_0$ the set of elementary tensors on which
one has cobounded $\phi$ at the previous step, and by $E_1$ the set
on which one wishes to cobound, then $(ds+sd)(E_0)\subset (E_0)$ and
$\et - (ds+sd)(\et) \in E_0$ for $\et \in E_1$. This gives us the
result: for, given $\phi\in \ZC^n(\cA)$ such that
$\phi_0=\phi-\delta(\psi)$ vanishes on $E_0$, defining $\psi_1$ on
$E_1$ by $\psi_1(\et)=\phi_0(s(\et))$ (for $\et \in \CC_{n-1}(\cA)$)
gives
\begin{eqnarray*}
(\phi_0 - \delta\psi_1)(\et)&=&\phi_0(\et) - \psi_1(d\et)\\
    &=&\phi_0(\et) - \phi_0(sd\et)\\
    &=&\phi_0(\et-(ds+sd)\et),
\end{eqnarray*}
which still vanishes on $E_0$ and now vanishes on $E_1$. Note that
if $E_0 \subset E_1$, then it  is sufficient to verify that
 $\et -(ds+sd)(\et) \in E_0$ for $\et \in E_1$.

\end{section}

\begin{section}{A first normalization step}
\begin{subsection}{Cobounding cyclically with norm control}\label{ss:using-YCnote}
Our first observation is that, if $R$ is a rectangular band, then $\ell^1(R)$ is a $1$-biprojective Banach algebra: that is, there exists\footnotemark\ an $\ell^1(R)$-bimodule map $\sigma:\ell^1(R)\to \ell^1(R)\ptp\ell^1(R)$, which has norm $1$ and which is right inverse to the product map $\ell^1(R)\ptp\ell^1(R)\to\ell^1(R)$.
\footnotetext{In fact, we can take $\sigma(x)=xe\tp ex$ where $e$ is a fixed element in $R$, although we do not need to know this for the arguments which follow.}
(To see how this definition relates to the original homological one, see \cite[Ch.~IV, \S5]{He0}.)

As observed in \cite[Lemma 7.5]{YC_SA}, one can use $\sigma$ to construct a splitting homotopy for the simplicial chain complex, and thus show directly that $\HH^n(\ell^1(R))=0$ for all $n\geq 1$. The corresponding result for cyclic cohomology is more complicated, but can nevertheless be deduced using the Connes-Tzygan long exact sequence for Banach algebras, \cite{He1}.

Now consider a general band $S$ which decomposes into rectangular band components as $S=\coprod_\al R_\al$.
We wish to use this decomposition to reduce our cohomology problem to the case of rectangular bands. Although we are ultimately interested in simplicial cohomology, it seems necessary at certain points in our reduction technique to be working with cyclic cochains. Thus, we shall need to consider the \emph{cyclic} cohomology of $\ell^1(R_\alpha)$, for each $\alpha\in L$.
Since we need to deal with all the $R_\alpha$ simultaneously, it no longer suffices to appeal to~\cite[Theorem 25]{He1}; a more precise version of that result is needed, as follows.

\begin{thm}\label{t:unifCC-biflat}
Let $A$ be a biflat Banach algebra, with biflatness constant $K\geq 1$\/.
Let $m\geq 0$\/.
\begin{YCnum}
\item For every $\psi\in \ZC^{2m+1}(A)$ there exists $\chi\in \CC^{2m}(A)$ such that $\psi=\dif\chi$\/; moreover, $\chi$ may be chosen to satisfy the bound
\[ \norm{\chi}\leq 2(m+1)^3 K^{4m}\norm{\psi}.\]
\item For every $\psi\in \ZC^{2m+2}(A)$ there exists $\chi\in \CC^{2m+1}(A)$
 and $\tau\in \CC^0(A)$
 such that
$\psi=\tau^{(2m+2)}+\dif\chi$\/;
moreover, $\tau$ and $\chi$ may be chosen to satisfy the bounds
\[ \norm{\tau}\leq K^{2(m+1)}\norm{\psi},\quad \norm{\chi}\leq 2(m+1)^3 K^{2(2m+1)}\norm{\psi}.\]
\end{YCnum}
\end{thm}

Theorem~\ref{t:unifCC-biflat} may well be implicitly known to specialists; a fairly direct and self-contained proof can be found in~\cite{YC_CTconst}.
The important aspect, for our purposes, is that the constants which control the cobounding depend only on the degree of the cocycle and on the biflatness constant~$K$.
\end{subsection}

\begin{subsection}{Initializing a cyclic cocycle on $\ell^1(S)$}
\label{ss:initializing}

Given a cyclic cocycle $\psi$, we are trying to find a cyclic cochain $\chi$ such that $\psi-\dif\chi$ vanishes on a conveniently large set. This will be done in stages: the precise definition for our first step is as follows.

\begin{Definition}
Let $\phi\in\ZC^n(\ell^1(S))$. We say that $\phi$ is \dt{rectangular-band-normalized}, or \dt{$R$-normalized} for short, if it vanishes on $\et = \chain$ whenever all $x_i$ are in the same rectangular band component of $S$.
\end{Definition}

\begin{lem}\label{l:traces}
\
\begin{YCnum}
\item If $R$ is a rectangular band, every continuous trace on $\ell^1(R)$ is a scalar multiple of the augmentation character $\epsilon: \ell^1(R) \to \Cplx$ where $\epsilon(\sum_{s\in R} \lambda_s s) = \sum \lambda_s$.
\item If $S$ is an arbitrary band, decomposed canonically as $\coprod_{\al\in L} R_\al$\/, then $\SZ{0}{\ell^1(S)}$ is isomorphic to $\ell^\infty(L)$.
\end{YCnum}
\end{lem}

\begin{proof}[Proof sketch]
Part (i) is proved by fixing $e\in R$ and noting that any trace $\tau$ on $\ell^1(R)$ must satisfy $\tau(x)=\tau(xeex) = \tau(exxe) = \tau(e)$ for all $x\in R$.
Part (ii) follows by considering the restriction of a trace $\tau\in \SZ{0}{\ell^1(S)}$ to each subalgebra $\ell^1(R_\al)$.
\end{proof}


\begin{prop}[Normalization on each rectangular component]\label{p:initialization}
Let $n\geq 1$\/.
\begin{YCnum}
\item For every $\psi\in \ZC^{2n-1}(\ell^1(S))$, there exists $\chi\in \CC^{2n-2}(\ell^1(S))$ such that $\psi-\dif\chi$ is $R$-normalized.
\item For every $\psi\in \ZC^{2n}(\ell^1(S))$, there exists $\tau \in \SZ{0}{\ell^1(S)}$ and $\chi\in \CC^{2n-1}(\ell^1(S))$ such that $\psi-\tau^{(2n)} -\dif\chi$ is $R$-normalized.
\end{YCnum}
\end{prop}

\begin{proof}
We recall that for each $\al$\/, the algebra $\ell^1(R_\al)$ is biprojective with constant~$1$\/.

The case of odd degree is straightforward. Given $\psi\in\ZC^{2n-1}(\ell^1(S))$, let $\psi_\al$ denote the restriction of $\psi$ to the subalgebra $\ell^1(R_\al)$\/; then by Theorem~\ref{t:unifCC-biflat}, for each $\al$ there exists $\chi_\al\in \CC^{2n-2}(\ell^1(R_\al))$ such that $\dif\chi_\al=\psi_\al$ and
\[ \norm{\chi_\al}\leq 2n^3 \norm{\psi_\al}\leq K'_n \norm{\psi},\]
\emph{where the constant $K'_n$ does not depend on $\al$}\/.
Given a $(2n-1)$-tuple $(x_1,\ldots, x_{2n-1})\in S^{2n-1}$\/, we define
\[
\chi(x_1,\ldots, x_{2n-2})(x_{2n-1}) \defeq
\left\{ \begin{aligned}
\chi_\al(x_1,\ldots, x_{2n-2})(x_{2n-1}) & \quad\text{ if $x_1,\ldots, x_{2n-1}\in R_\al$ for some $\al\in L$\/,} \\
0 & \quad\text{ otherwise.}
\end{aligned}\right.
\]
Then $\chi$ extends to a bounded $(2n-1)$-multilinear functional on $\ell^1(S)$\/, which is clearly a cyclic cochain since each $\chi_\al$ is. By construction, if $x_1,\ldots, x_{2n}\in R_\al$ for some $\al$\/, then
\[ (\psi-\dif\chi)(x_1,\ldots, x_{2n-1})(x_0) = (\psi_\al-\dif\chi_\al) (x_1,\ldots, x_{2n-1})(x_{2n}) = 0,\]
and thus $\psi-\dif\chi$ is $R$-normalized.

The case of even degree is similar, except that we have to deal with cocycles arising from traces.
As before, let $\psi\in\ZC^{2n}(\ell^1(S))$\/, and for each $\al$ let $\psi_\al\in \ZC^{2n}(\ell^1(R_\al))$ be the restriction of $\psi$ to $\ell^1(R_\al)$ in each variable.
By Theorem~\ref{t:unifCC-biflat}, for each $\al$ there exists $\chi_\al\in \CC^{2n-1}(\ell^1(R_\al))$ and $\tau_\al\in \ZC^0(\ell^1(R_\al))$ such that $\dif\chi_\al+\tau_\al^{(2n)}=\psi_\al$ with
\[ \norm{\tau_\al}\leq K''_n\norm{\psi} \quad\text{ and }\quad\norm{\chi_\al}\leq K''_n \norm{\psi},\]
\emph{where the constant $K''_n$ does not depend on $\al$}\/.

By Lemma~\ref{l:traces} each $\tau_\al$ is constant, with value~$c_\al$ say. Let $\tau: S\to \Cplx$ be defined by $\tau:R_\al \to \{c_\al\}$; then $\tau$ is a bounded trace on $\ell^1(S)$\/, and the restriction of $\tau^{(2n)}$ to $\ell^1(R_\al)$ is clearly just $\tau_\al^{(2n)}$\/.

Also, given a $2n$-tuple $(x_1,\ldots, x_{2n})\in S^{2n}$\/, we define
\[
\chi(x_1,\ldots, x_{2n-1})(x_{2n}) \defeq
\left\{ \begin{aligned}
\chi_\al(x_1,\ldots, x_{2n-1})(x_{2n}) & \quad\text{ if $x_1,\ldots, x_{2n}\in R_\al$ for some $\al\in L$\/,} \\
0 & \quad\text{ otherwise.}
\end{aligned}\right.
\]
Then $\chi$ extends to a well-defined cyclic $(2n-1)$-cochain on $\ell^1(S)$\/, and by construction we find that, for each $\al$ and every $x_1,\ldots, x_{2n+1}\in R_\al$\/,
\[  (\psi-\tau^{(2n)}-\dif\chi)(x_1,\ldots, x_{2n})(x_{2n+1})
  = (\psi_\al-\tau_\al^{(2n)}-\dif\chi_\al)(x_1,\ldots, x_{2n})(x_{2n+1})
\]
as required.
\end{proof}
\end{subsection}

\end{section}

\begin{section}{A sufficient condition for using the Connes-Tzygan sequence}
\label{s:h-unital}
\begin{dfn}
Given a Banach algebra $B$, the \dt{reduced Hochshild complex} $\CR_*(B)$ is the following chain complex of Banach spaces:
\[ 0 \longleftarrow B \lto{d} B^{\ptp 2} \lto{d} \cdots \lto{d} B^{\ptp n} \lto{d} \cdots \]
where the boundary map $d$ is defined by
\[ d(b_1\tp \cdots\tp b_{n+1}) = \sum_{j=1}^n (-1)^j \filler{j-1}\tp b_jb_{j+1} \tp \filler{n-j} \qquad(b_1,\dots, b_{n+1}\in B). \]
For each $n\geq 0$, we write $\ZR_n(B)$ for the kernel of $d:B^{\ptp n+1}\to B^{\ptp n}$ and $\BR_n(B)$ for the image of $d:B^{\tp n+2}\to B^{\ptp n+1}$\/.
\end{dfn}

Consider the case $B=\ell^1(S)$.
In order to construct the Connes-Tzygan long exact sequence for $\ell^1(S)$, we need to know that the complex $\CR_*(\ell^1(S))$ is exact, i.e.~that $\BR_n(\ell^1(S))=\ZR_n(\ell^1(S))$ for all $n\geq 0$ (see \cite[Theorem 11]{He1}). This is the goal of the current section.

\begin{rem}
Even in the case where $B=L$, i.e.~when our band is a semilattice, the result is not immediately obvious; hitherto, the only known proof used a special case of the main results in~\cite{Ch1}.
\end{rem}

The case $n=0$ is trivial, since $\ZR_0(\ell^1(S))=\ell^1(S)$, and for $a=\sum_{s\in S} \lambda_s s\in\ell^1(S)$ we have $d\left(\sum_{s\in S}\lm_s \lcu{s}\tp s \right) = a$. We therefore restrict attention in what follows to the case $n\geq 1$.

Left-coherent units will play a key role in our proof. In particular we need the following lemma in several places.
\begin{Lemma}\label{l:units-trick}
Let $y,z\in S$. Then $\lcu{\lcu{y}}=\lcu{y}$ and $\lcu{y}\lcu{yz}=\lcu{yz}$.
\end{Lemma}

\begin{proof}
Both identities follow from our explicit construction of the function $\lcu{\cdot}$. They can also be deduced from the coherence properties that were observed earlier.

The first identity follows since $\lcu{x}=\lcu{xy}$ for all
$x\in R_\alpha$ and $y\in R_\beta$ with $\alpha\preceq\beta$\/, so that taking $x=\lcu{y}$ does the job.
For the second identity, note that
$x\lcu{x} = (\lcu{x} x) \lcu{x} = \lcu{x}$,
the last equality following because $x$ and $\lcu{x}$ lie in the same rectangular band. In particular, taking $x=yz$ yields
\[ \lcu{y}\lcu{yz} = \lcu{y}(yz\lcu{yz}) = yz\lcu{yz} = \lcu{yz} \]
as required.
\end{proof}

For $1\leq k \leq n+1$, let $s_k:\ell^1(S^{n+1})\to \ell^1(S^{n+2})$ be defined by
\begin{equation}\label{eq:dfn-inserts}
s_k(\subchain{1}{n+1}) = (-1)^k \subchain1{k-1}\tp \lcu{x_k}\tp \subchain{k}{n+1}\,.
\end{equation}
Then put
\[ Q_k\defeq ds_k+s_kd-I:\ell^1(S^{n+1})\to \ell^1(S^{n+1}).\]
If $z\in\ZR_n(\ell^1(S))$, then a straightforward calculation shows that $Q_k\cdots Q_1(z)$ is homologous to $(-1)^k z$ (that is, the two tensors differ by an element of $\BR_n$). The work lies in obtaining a formula for $Q_n\cdots Q_1$, which will allow us to see that $Q_n\cdots Q_1(z)$ is homologous to zero.

\begin{rem}
We give some motivation for the introduction of the maps $s_k$ and $Q_k$, and the attention paid to $Q_n\dotsb Q_1$\/.
When $B$ is a Banach algebra with identity~$\id$, it is well known that the chain complex $\CR_*(B)$ is exact, and that one can construct an explicit contracting homotopy~$\sigma$. The maps $\sigma_{n+1}: B^{\ptp n+1}\to B^{\ptp n+2}$ are given by $\sigma(b_1\tp \cdots\tp b_{n+1}) = - \id \tp b_1\tp\cdots\tp b_{n+1}$\/, and one can check directly that $d\sigma+\sigma d = I$.

In the present case, $\ell^1(S)$ might not even have a bounded approximate identity. Nevertheless, since we do have local left-coherent units $\lcu{x}$,
it is natural to see how far $ds_1+s_1d$ is from being the identity map, i.e.~how far $Q_1$ is from being zero.

Although $Q_1(\et)$ is in general non-zero (see~\eqref{eq:en-avance} below), it is a tensor with more `structure' than $\et$ in some sense; and when we successively apply the maps $Q_2$, \dots, $Q_n$,
at each stage we increase the amount of structure present.
Thus, given $z\in \ZR_n(\ell^1(S))$, the tensor $Q_n\dotsb Q_1(z)$, which as already remarked is homologous to $(-1)^{n}z$\/, will be so highly structured that it falls into $\BR_n(\ell^1(S))$.
\end{rem}

\medskip
Let us start our argument by calculating $Q_1(\et)$ where $\et=x_1\tp \cdots\tp x_{n+1}$\/.
Since
\[ \begin{aligned}
 ds_1(\et) & =  x_1\tp\ldots\tp  x_{n+1}  &+&\quad \sum_{j=1}^{n}(-1)^j \lcu{x_1}\tp x_1\tp \cdots\tp x_jx_{j+1}\tp\filler{n-j} \\
\text{and}\quad
 s_1d(\et) & = \lcu{x_1x_2}\tp x_1x_2\tp\filler{n-1}
	 &+&\quad \sum_{j=2}^{n} (-1)^{j-1} \lcu{x_1}\tp x_1\tp \cdots\tp x_jx_{j+1}\tp\filler{n-j}\,,
\end{aligned}  \]
we have
\begin{equation}\label{eq:en-avance}
Q_1(\et) = - \lcu{x_1}\tp x_1x_2\tp \filler{n-1} \quad+\quad \lcu{x_1x_2} \tp x_1x_2 \tp\filler{n-1}\,.
\end{equation}
Note that $\ran(Q_1)$ is contained in the kernel of the map
\[ \rho_1: y_1\tp \cdots\tp y_{n+1} \mapsto y_1\lcu{y_2}\tp y_2 \tp\filler{n-1}\,. \]

Now let $\yy=y_1\tp\ldots\tp y_{n+1}$ where $y_1,\dots,y_{n+1}\in S$.
A similar calculation to that for $Q_1$ shows that
\[ Q_2(\yy) = \left\{
	\begin{aligned}
	- \rho_1(\yy)
	\quad & \quad - y_1\tp \lcu{y_2}\tp y_2y_3 \tp \filler{n-2} \\
	- y_1y_2\tp \lcu{y_3}\tp y_3\tp \filler{n-2}
	\quad & \quad + y_1\tp \lcu{y_2y_3}\tp y_2y_3 \tp \filler{n-2} \\
	\end{aligned}\right. \]
Then, since $\rho_1Q_1(\et)=0$, we find after some calculation that
\begin{equation}
\begin{aligned}
 Q_2Q_1(\et)
	 = & \quad \lcu{x_1}\tp\lcu{x_1x_2}\tp x_1x_2x_3 \tp \filler{n-2}
	   & -\quad &  \lcu{x_1}\tp\lcu{x_1x_2x_3}\tp x_1x_2x_3 \tp \filler{n-2} \\
	   & - \lcu{x_1x_2}\tp\lcu{x_1x_2}\tp x_2x_2x_3 \tp \filler{n-2}
	   & +\quad &\lcu{x_1x_2}\tp\lcu{x_1x_2x_3}\tp x_1x_2x_3 \tp \filler{n-2} .\\
\end{aligned}
\end{equation}
By Lemma~\ref{l:units-trick}, $\lcu{x_1 x_2}\lcu{x_1 x_2 x_3}=\lcu{x_1 x_2 x_3}\lcu{x_1 x_2 x_3}$\/: hence $Q_2Q_1(\et)$ lies in the kernel of the map
\[ \rho_2: x_1\tp \cdots\tp x_{n+1} \mapsto x_1\tp x_2\lcu{x_3}\tp x_3\tp \filler{n-3} \,. \]
It also lies in $\ker(\rho_1)$, as can be shown using Lemma~\ref{l:units-trick} again.



\medskip
By continuing in this way, one could calculate $Q_n\dotsb Q_1(\et)$ directly; but it would become harder to keep track of the terms involved and how they cancel.
To do the necessary book-keeping, we write the boundary operator as an alternating sum of \dt{face maps}.
That is, for $n\geq 1$ and $1\leq i\leq n$, let $\face_i: \ell^1(S^{n+1}) \to \ell^1(S^{n})$ denote the map defined by
\[ \face_i: \subchain1{n+1} \mapsto \filler{i-1} \tp x_ix_{i+1}\tp \filler{n-i}\,, \]
so that $d = \sum_{i=1}^{n} (-1)^i \face_i: \ell^1(S^{n+1}) \to \ell^1(S^{n})$.

The following identities are easily verified by checking on elementary tensors.
\begin{subequations}
\begin{align}
\label{eq:initial-terms}
\face_i s_k & = - s_{k-1}\face_i & \qquad\text{ if $i+2 \leq k\leq n+1$} \\
\label{eq:in-and-out}
\face_i s_i & = (-1)^i I \\
\label{eq:tail-terms}
\face_i s_k & = s_k \face_{i-1} & \qquad\text{ if $k+2\leq i \leq n+1$.}
\end{align}
\end{subequations}
Using these identities, for $1\leq k \leq n$ we may rewrite $Q_k$ as

\begin{equation}\label{eq:expand-Q}
\begin{aligned}
Q_k & = \sum_{i=1}^{n+1} (-1)^i \face_i s_k + \sum_{i=1}^{n} (-1)^i s_k \face_i - I  \\
    & = \sum_{i=1}^{k+1} (-1)^i \face_i s_k + \sum_{i=1}^k  (-1)^i s_k \face_i - I
    & \quad\text{(by \eqref{eq:tail-terms})} \\
& =
	\sum_{i=1}^{k-1} (-1)^i \face_i s_k + (-1)^{k+1}\face_{k+1}s_k
	 + \sum_{i=1}^k(-1)^i s_k \face_i & \quad(\text{by \eqref{eq:in-and-out}}) \\
& =
	\sum_{i=1}^{k-2} (-1)^{i+1} s_{k-1}\face_i  - \rho_{k-1}
 + (-1)^{k+1}\face_{k+1}s_k  + \sum_{i=1}^k(-1)^i s_k \face_i
 & \quad(\text{by \eqref{eq:initial-terms}})
\end{aligned}
\end{equation}

\noindent
where $\rho_i = (-1)^{i+1} \face_i s_{i+1}$, i.e.
\begin{equation}\label{eq:define-rho}
\rho_i(\subchain1{n+1}) = \filler{i-1}\tp x_i\lcu{x_{i+1}}\tp x_{i+1}\tp\filler{n-i}\,.
\end{equation}
Note that
$\face_i\rho_i = \face_i$\/, which
implies that $\ker\rho_i\subseteq \ker\face_i$.

Next, for $1\leq k \leq n$, let
$\widetilde{Q}_k \defeq (-1)^{k+1} \face_{k+1} s_k + (-1)^k s_k\face_k$, so that
\begin{equation}
\widetilde{Q}_k(\et) =
 - \filler{k-1}\tp \lcu{x_k}\tp x_kx_{k+1}\tp \filler{n-k}
   \quad + \quad
   \filler{k-1} \tp \lcu{x_kx_{k+1}}\tp x_kx_{k+1}\tp \filler{n-k} \,.
\end{equation}
The point of this definition is its use in the following result.

\begin{prop}\label{p:creep}
Let $1\leq r \leq n$. Then
\begin{YCnum}
\item $Q_r\dotsb Q_1 = \widetilde{Q}_r \dotsb\widetilde{Q}_1$\/;
\item $\ran(Q_r\dotsb Q_1) \subseteq (\ker\rho_1)\cap\dots\cap(\ker\rho_r)$\/.
\end{YCnum}
\end{prop}

\begin{proof}
The proof is by induction. When $r=1$, part (i) is trivial and part (ii) was proved above, see Equation~\eqref{eq:en-avance}. Suppose both parts hold true for $r=k-1$, where $2\leq k\leq n$. Then, since
$\ker\rho_i\subseteq \ker\face_i$ for all $i$, and since (ii) holds for $r=k-1$,
\[
 \face_i (Q_{k-1}\dotsb Q_1) = 0
 \quad\text{ for $1\leq i \leq k-2$, and}\quad
\rho_{k-1} (Q_{k-1}\dotsb Q_1) =0. 
 \]

\noindent
Comparing this with \eqref{eq:expand-Q}, we see that ${Q_k}Q_{k-1}\dotsb Q_1=\widetilde{Q_k}Q_{k-1}\dotsb Q_1$; and since (i) holds for $r=k-1$, we have $Q_k\dotsb Q_1 = \widetilde{Q}_k \dotsb\widetilde{Q}_1$\/. Thus part (i) holds for $r=k$.

To complete the inductive step, we must show (ii) holds for $r=k$, i.e.~that
 $\widetilde{Q}_k(\ker\rho_i)\subseteq\ker\rho_i$ for all $1\leq i \leq k$.

For $1\leq i \leq k-2$ this is straightforward, since a direct check on elementary tensors shows that $\rho_i$ commutes with $\widetilde{Q}_k$.
For $i=k-1$, by using Lemma~\ref{l:units-trick} we obtain

\begin{eqnarray*}
\rho_{k-1}\widetilde{Q}_k(\et)
& = & \rho_{k-1}\left( -\filler{k-2} \tp x_{k-1}\tp\lcu{x_k}\tp x_kx_{k+1}\tp \filler{n-k}\right. \\
&   & \quad\qquad \left. + \filler{k-2} \tp x_{k-1}\tp\lcu{x_kx_{k+1}}\tp x_kx_{k+1}\tp \filler{n-k}\right) \\
& = & - \filler{k-2} \tp x_{k-1}\lcu{x_k} \tp \lcu{x_k}\tp x_kx_{k+1}\tp\filler{n-k} \\
&   & + \filler{k-2} \tp x_{k-1}\lcu{x_kx_{k+1}} \tp\lcu{x_kx_{k+1}}\tp x_kx_{k+1}\tp\filler{n-k} \\
& = & \rho_{k-1}\left( -\filler{k-2} \tp x_{k-1}\lcu{x_k} \tp \lcu{x_k}\tp x_kx_{k+1}\tp \filler{n-k} \right. \\
&   & \quad\qquad + \left. \filler{k-2} \tp x_{k-1}\lcu{x_k} \tp\lcu{x_kx_{k+1}}\tp x_kx_{k+1}\tp \filler{n-k} \right) \\
& = & \rho_{k-1}\widetilde{Q}_k \left(\filler{k-2} \tp x_{k-1}\lcu{x_k} \tp x_k \tp x_{k+1}\tp \filler{n-k} \right)
\qquad =  \rho_{k-1}\widetilde{Q}_k \rho_{k-1} (\et).
\end{eqnarray*}

\noindent
Finally, another direct calculation on elementary tensors, using Lemma~\ref{l:units-trick}, shows that $\rho_k\widetilde{Q}_k=0$\/. This completes the inductive step.
\end{proof}

\begin{lem}\label{l:coup-de-grace}
$(s_nd+ds_{n+1}-I)\widetilde{Q}_n=0$.
\end{lem}

\begin{proof}
Using the identities \eqref{eq:initial-terms} and \eqref{eq:in-and-out}, we have
\[ \begin{aligned}
(s_nd+ds_{n+1}-I)(\yy)
 & = \sum_{j=1}^{n}(-1)^j s_n\face_j(\yy) + \sum_{k=1}^{n+1} (-1)^k \face_k s_{n+1}(\yy) - \yy \\
 & = (-1)^{n}s_n\face_n(\yy) + (-1)^{n}\face_ns_{n+1}(\yy)  \\
 & = \quad \filler{n-1} \tp \lcu{y_ny_{n+1}}\tp y_ny_{n+1}
     \quad - \quad \filler{n-1} \tp y_n\lcu{y_{n+1}} \tp y_{n+1}\,.
\end{aligned} \]

Thus
\[ \begin{aligned}
 & (s_nd+ds_{n+1}-I)\widetilde{Q}_n(\et) \\
 & =  - (s_nd+ds_{n+1}-I)\left(\filler{n-1} \tp \lcu{x_n} \tp x_nx_{n+1}\right)  \\
 & \quad  + (s_nd+ds_{n+1}-I)\left(\filler{n-1} \tp \lcu{x_nx_{n+1}} \tp x_nx_{n+1}\right)  \\
 &   =   - \filler{n-1} \tp \lcu{x_nx_{n+1}} \tp x_nx_{n+1}
	\quad + \quad
	   \filler{n-1} \tp \lcu{x_n} \lcu{x_nx_{n+1}}\tp x_nx_{n+1} \\
 & \quad + \filler{n-1} \tp \lcu{x_nx_{n+1}} \tp x_nx_{n+1}
	\quad - \quad
	   \filler{n-1} \tp\lcu{x_nx_{n+1}}\lcu{x_nx_{n+1}}\tp x_nx_{n+1}\,.
\end{aligned} \]
Since
$\lcu{x_n}\lcu{x_nx_{n+1}}=\lcu{x_nx_{n+1}} = \lcu{x_nx_{n+1}}\lcu{x_nx_{n+1}}$, by Lemma~\ref{l:units-trick}, these four terms cancel pairwise to give~$0$.
\end{proof}

\begin{thm}\label{t:we-have-hunital}
The complex $\CR_*(\ell^1(S))$ is exact.
\end{thm}

\begin{proof}
Let $n\geq 1$\/. As already mentioned, it suffices to prove that $\ZR_n(\ell^1(S))=\BR_n(\ell^1(S))$. Thus, let $z\in \ZR_n(\ell^1(S))$. A simple induction using the definition of the $Q_i$ shows that, for $1\leq k\leq n$,
\[ Q_k\dotsb Q_1(z) = (ds_k-I)\dotsb(ds_1-I)(z)\in \ZR_n(\ell^1(S)). \]
Hence
\[ (s_nd+ds_{n+1}-I)Q_n\dotsb Q_1(z) = (ds_{n+1}-I)(ds_n-I)\dotsb(ds_1-I)(z). \]
Now, combining Proposition~\ref{p:creep} and Lemma~\ref{l:coup-de-grace} yields
\[ (s_nd+ds_{n+1}-I)Q_n\dotsb Q_1 = (s_nd+ds_{n+1}-I)\widetilde{Q}_n\dotsb \widetilde{Q}_1 = 0 \,.\]
Thus $(ds_{n+1}-I)(ds_n-I)\dotsb(ds_1-I)(z)=0$, and expanding out we deduce that $z\in \BR_n(\ell^1(S))$, as required.
\end{proof}

\begin{rem}
Even if the present work is focused on band semigroups, it should nevertheless be noted that the calculations of this section apply equally well to a Clifford semigroup.

For present purposes (cf.~\cite[Theorem IV.2.1]{Ho}), we say that a semigroup $\bbG$ is a \dt{Clifford semigroup} if it decomposes as a disjoint union $\bbG=\coprod_{\al\in L} G_\al$ of subsemigroups, where the indexing set $L$ is a semilattice, each $G_\al$ is a group with identity element $e_\al$, and $G_\al\cdot G_\beta \subseteq G_{\al\beta}$ for all $\al,\beta$\/. We can define an analogous local left unit function $\lcu{\cdot}: \bbG \to \bbG$, which sends $x\in G_\al$ to $e_\al$\/.

If we were then to repeat the calculations of this section, we would find that everything goes through (with slight simplifications, in fact), and would thus obtain a direct proof that the complex $\CR_*(\ell^1(\bbG))$ is exact. This implies, for instance, that we have a Connes-Tzygan long exact sequence for $\ell^1(\bbG)$, so that the results of \cite{YC_SA} for the simplicial cohomology of $\ell^1(\bbG)$ could be applied to obtain results for its \emph{cyclic} cohomology.
\end{rem}

\end{section}

\begin{section}{Inductively reducing down to the $R$-normalized case}

\begin{subsection}{Minimal elements}\label{s:with-minimum}

\begin{Definition}[Degree of elements and tensors]
\label{dfn:degree}
For $a \in S$, let $[a]$ denote its \dt{degree} in $S$, that is if $a\in R_\alpha$ then $[a]\defeq\alpha$.
Then, given an elementary (sub)tensor of point masses $w=\subchain{k}{l}$, define the \dt{degree of $w$}, denoted by $[w]$, to be $[ x_k x_{k+1} \cdots x_{l}]$.
\end{Definition}

In this section, we shall prove that one can cobound any $R$-normalized $\phi\in \ZC_n(\cA)$ on those elementary tensors $\et=\chain$ such that $[x_i]=[\et]$ for some $i$.

We work cyclically with indices when dealing with cyclic cohomology: for instance, the interval $[n-1, 2]$ is
the set $\{n-1, n, n+1, 1, 2\}$ and we call $\subchain{n-1}{n+1}
\otimes x_1 \otimes x_2$ a subtensor. We will sometimes emphasize this
by describing these as \dt{cyclic intervals} or \dt{cyclic subtensors}.

\begin{Definition} An elementary tensor $\et=\chain$ has a \dt{minimal element} $x_i$, for some $i$, 
if $[x_i]=[\et]$. (The degree of such an element is a minimum.) A \dt{minimal block} is a (cyclic)
subtensor $\subchain k{l}$ such that $[x_i] = [\et]$ for all $i$ in
the cyclic interval $[k,\, l]$, and $x_{k-1} \ne [\et] \ne x_{l+1}$.
\end{Definition}

Note that if $\et$ is a minimal block itself, then all $x_i$ are in
the same rectangular band; the assumption that $\phi$ is $R$-normalized therefore implies that it vanishes
on such an $\et$.

For elementary tensors with at least one minimal element, let
\[ J_\et = \{ i \in \{1,2,\ldots,n+1\} \;:\; \text{$x_i$ is the first component of a minimal-block.}\}. \]
$J_\et$ is the set of the indices of all \emph{initial} points of minimal blocks.
%
Given $i\in J_\et$, define $s_i :\Ho{C}{n}(\cA)\to\Ho{C}{n+1}(\cA)$ by
\[ s_i(\chain) = (-1)^{i} (\subchain 1{i-1}\otimes \lcu{x_i}\otimes \subchain i{n+1})\,,\]
and then define $s : \Ho{C}{n}(\cA)\to\Ho{C}{n+1}(\cA)$ on $\et$ with
$j$ minimal-blocks by

\begin{equation}
s(\et) = \frac {1}{j} \sum_{i\in J_\et} s_i(\et).
\end{equation}
If there are no minimal blocks, ($J_\et$ is empty), set $s({\bf
x})= 0$.

Dualizing this operator yields $\sigma :\SC{n+1}{\cA}\to\SC{n}{\cA}$, defined on $\phi\in\SC{n+1}{\cA}$ by
\begin{equation}
\sigma\phi(\et) = \phi (s(\et)).
\end{equation}
We now wish to show that $\sigma$ takes cyclic cochains to cyclic cochains.

\begin{Lemma}\label{l:our-first-mending-is-cyclic}
If $\phi$ is cyclic then so is $\sigma\phi$.
\end{Lemma}

\begin{proof}
The key point is that the definition of minimal blocks is equivariant
with respect to cyclic shifts, that is,
\[ i\in I_{(\chain)} \iff i+1 \in I_{(x_{n+1}\otimes \subchain 1{n}).}
\]
where this is understood cyclically in the case $i=n+1$\/.

If $i\in I_{(\chain)}$ and $1\leq i \leq n$, then
\[ \begin{aligned}
\cyc s_i(\chain)
 & = (-1)^i \cyc(
{\subchain 1{i-1} \otimes \lcu{x_i} \otimes \subchain i{n+1}}
) \\
 & = (-1)^{n+1}(-1)^i (x_{n+1} \otimes
\subchain{1}{i-1} \otimes \lcu{x_i} \otimes \subchain{i}{n}
) \\
 & = (-1)^n s_{i+1}(x_{n+1}\otimes \subchain 1{n}).
\end{aligned} \]
On the other hand, if $n+1\in I_{\chain}$, then
\[ \begin{aligned}
\cyc^2 s_{n+1}(\chain)
 & = (-1)^{n+1} \cyc^2(x_1\tp \cdots\tp\lcu{x_{n+1}}\tp x_{n+1}) \\
 & = (-1)^{n+1}(\lcu{x_{n+1}}\tp x_{n+1} \otimes \subchain 1{n}) \\
 & = (-1)^n s_1(x_{n+1}\tp\subchain1n).
\end{aligned} \]
Thus, if $\psi\in\CC^{n+1}(\cA)$, so that $\psi\circ\cyc=\psi$, we find that
\[ \begin{aligned}
\cyc\sigma\psi(\chain)
 & = (-1)^n \psi(s(x_{n+1}\tp\subchain1n)) \\
 & = \psi(s(\chain))
   = \sigma\psi(\chain)
\end{aligned} \]
as required.
\end{proof}

\begin{Proposition} \label{minimal in sight}
For any $R$-normalized $\phi\in \ZC^n(\cA)$, there exists $\psi \in\CC^{n-1}(\cA)$,
such that $(\phi-\delta\psi)(\et)=0$ for all elementary tensors $\et$ with
some minimal element.
\end{Proposition}

\begin{proof} Let $\et=\chain$ be an tensor with some minimal element and let
$m_{\et}$ be the sum of the length of its minimal blocks.
Since we are working in \emph{cyclic} cohomology, we can by cycling our tensor assume without loss of generality that
$1\in J_\et$.
It will be
convenient to write $\et = u_1 \otimes v_1 \otimes \cdots \otimes
u_l \otimes v_l$ where $u_l$, $l=1, \dots, j$ are the $j$ minimal
blocks.

If all elements are minimal, we say that $m_{\et}=n+1$; then since $\phi$ is $R$-normalized it will be assumed to vanish on~$\et$.

Suppose we can cobound on $\et$ such that $m_\et\ge K$. Let $\et$ be
such that $m_\et = K-1$ and consider $(ds+sd)\et$. In the notation
of \eqref{eq:dw1}, in $d(\et)$ there are terms with $d_c(v_l)$, $u_l
\cdot v_l$ and $v_l \cdot u_{l+1}$ which all have $K-1$ minimal
elements, and those with $d_c(u_l)$ which have $K-2$ minimal
elements. Applying $s$ increases the number of minimal elements by
one in all terms, and therefore by induction it suffices to
consider only those terms in $sd(\et)$ of the form
\[  (-1)^i u_1 \otimes
v_1 \otimes \cdots \otimes \lcu{x_i} \otimes d_c(u_l) \otimes
v_l\otimes \cdots \otimes  u_j \otimes v_j\]
where $x_i$ is the first element
of $u_l$ (and where we have used that if $[x]\preceq[y]$ then
$\lcu{xy}=\lcu{x}$). Note that $d_c$ is applied to $u_l$ as a
subtensor of $\et$.

Similarly, in $ds(\et)$, we only need to consider terms of the form
\[  (-1)^i u_1 \otimes
v_1 \otimes \cdots \otimes d'_c( \lcu{x_i} \otimes u_l) \otimes
v_l\otimes \cdots \otimes u_j \otimes v_j,\]
where $d'_c$ is applied to
$\lcu{x_i} \otimes u_l$ as a subtensor of $s_i(\et)$: this effectively
changes the signs when comparing to terms in $sd(\et)$. When
summing, all terms cancel except
\[ u_1 \otimes
v_1 \otimes \cdots \otimes \lcu{x_i} \cdot u_l \otimes v_l\otimes \cdots
\otimes u_j \otimes v_j\]
which is $\et$.
\end{proof}
\end{subsection}

\begin{subsection}{Without minimal elements}\label{s:without-minimum}
The procedure for handling tensors without minimal elements is much more involved. Crucial to our construction
is the following definition.

\begin{Definition}
Let $\et = \chain \in S^{n+1}$ be without minimal element. We
say that a subtensor $\subchain k{l}$ has a \dt{minimal left element} if
$[x_k] \preceq [x_i]$ for all $i$ in the cyclic interval $[k,\, l]$.
A subtensor is a \dt{left-block} if it has a minimal left element
and is not strictly included in another subtensor which has a
minimal left element.
\end{Definition}

Clearly a tensor $\et\in S^{n+1}$ can have at most $n+1$ left-blocks.
Note that a tensor $\et$ with a minimal element informally corresponds to
having only one left-block: extending the definition to this case
leads to confusion as the initial element of such a left-block may
not be well defined. Nevertheless, if $\et$ doesn't have at least two
left-blocks, then it has a minimal element.

We stress again that we consider tensors like $x_n \otimes x_{n+1}\otimes x_1 \otimes x_2$ as subtensors of $\et$ and therefore as potential left-blocks.

\begin{notn}
For $2\leq j \leq n+1$, denote by $\cF_n^j$ the set of all elementary tensors in $S^{n+1}$ with at most $j$ left-blocks. We write $\cF_n^1$ for the subset of elementary tensors with a minimal element.
\end{notn}

These subsets give us a filtration $\cF_n^1\subset \cF_n^2\subset \cdots \subset \cF_n^{n+1}=S^{n+1}$, where $\cF_n^{n+1}$ has dense linear span in $\Ho{C}{n}(\cA)$\/.
Crucially, each face map $\face_i:\Ho{C}{n}(\cA)\to\Ho{C}{n-1}(\cA)$ cannot increase the number of left-blocks, and hence maps $\cF_n^j$ to $\cF_{n-1}^j$.

We have seen in the previous sections that if $\psi$ is an $R$-normalized $n$-cocycle, it is equivalent in cyclic cohomology to one that vanishes on $\cF_n^1$.

\begin{thm}\label{t:reduction}
Let $2 \leq j \leq n+1$ and let $\psi\in \ZC^n(\cA)$. Suppose that $\psi$ vanishes on $\cF_n^{j-1}$. Then it is equivalent in cyclic cohomology to a cocycle that vanishes on $\cF_n^j$.
\end{thm}

Theorem~\ref{t:reduction} will allow us, by an inductive argument, to conclude that if $\psi$ is an $R$-normalized cocycle, there exists a cyclic cochain $\phi$ such that $\psi=\dif\phi$. The proof of this theorem will take up the rest of this section and the following one.
\medskip
\begin{notn}

For elementary tensors without a minimal element, it is easy to see that any tensor $\et$ has a unique decomposition into left-blocks, and therefore we can define
\[ I_\et = \{ i \in
\{1,2,\ldots,n+1\} \;:\; \text{$x_i$ is the first component of a left-block.}\}. \]
$I_\et$ is the set of the indices of all \emph{initial} points of left-blocks.
\end{notn}


As in Section~\ref{s:with-minimum}, we shall now define an insertion operator in terms of this block structure.
(This operator will also be denoted by $s$, but this abuse of notation should not cause any confusion with the insertion operator that was considered in Section~\ref{s:with-minimum}.)
Given an elementary tensor $\et\in \Ho{C}{n}(\cA)$ and $i\in I_\et$, define
\[ s_i(\et) = (-1)^{i} \left(
\filler{i-1}
\otimes \lcu{x_i}\otimes x_i \tp
\filler{n-i}\right) \in\Ho{C}{n+1}(\cA)
.\]
and then define $s : \Ho{C}{n}(\cA)\to\Ho{C}{n+1}(\cA)$ by
\begin{equation}
s(\et) =  \sum_{i\in I_\et} s_i(\et).
\end{equation}
If there are no left-blocks, ($I_\et$ is empty), set $s(\et)= 0$.

Dualizing this operator, we define $\sigma :
\SC{n+1}{\cA}\to\SC{n}{\cA}$ on $\phi\in\SC{n+1}{\cA}$ by
\begin{equation}
\sigma\phi(\et) = \phi (s(\et)).
\end{equation}
The proof of Lemma \ref{l:our-first-mending-is-cyclic} shows that we
also have:

\begin{Lemma}\label{l:our-second-mending-is-cyclic}
If $\phi$ is cyclic then so is $\sigma\phi$.
\end{Lemma}

\medskip
Two parameters will be important for our approach in this section and the next: the \dt{degree} of a left-block; and the \dt{height} of a elementary tensor.
The degree of a (sub)tensor was defined earlier (Definition~\ref{dfn:degree}): note that the degree of a left-block will be the same as the degree of its initial element.

\begin{dfn}
If $T$ is a finite semilattice, and $\al\in T$, the \dt{height of $\al$ in $T$} is the length of the longest descending chain in $T$ which starts at $\al$. That is,
\[
 \high_{T}(\al) = \sup\{ m \;:\; \text{there exist $t_0,\dots, t_m\in T$ with
 $\al=t_m \succ t_{m-1} \succ\dots \succ t_0$} \}.
\]

If $\et=x_1\tp \cdots\tp x_{n+1}$ is an elementary tensor in $\Ho{C}{n}(\cA)$, let $L(\et)$ be the subsemilattice of $L$ that is generated by the set $\{ [x_1],\dots, [x_{n+1}]\}$, and define the \dt{height of $\et$} to be
\[ \high(\et) \defeq \sum_{i=1}^{n+1} \high_{L(\et)} ([x_i]). \]
\end{dfn}


Denote by $\cF_n^{j,h}$ the set of elementary tensors with at most $j$ left-blocks and with height at most~$h$.
Note for later reference that, if $\et$ has no minimum element, then there are crude bounds
\[ n+1 \leq \high(\et) \leq n(n+1). \]
We now define linear spaces which will be key to our induction:
\begin{equation}\label{eq:define_spaces}
G_{n, j, h}=\lin \cF_n^{j,h} + \lin \cF_n^{j-1}
\quad\text{ and }\quad
H_{n, j, h}=(I-\cyc)\Ho{C}{n}(\cA) + G_{n,j,h}.
\end{equation}

\begin{lem}
Let $T$ be a finite semilattice and let $F\subseteq T$ be a sub\-semilattice.
\begin{YCnum}
\item If $\al\in F$ then $\high_F(\al)\leq\high_T(\al)$\/;
\item If $\al,\beta\in T$ and $\al \prec \beta$ then $\high_T(\al)< \high_T(\beta)$.
\end{YCnum}
\end{lem}

The proofs of both parts are clear.

\begin{prop}\label{p:modified-main}
Let $\et=\subchain{1}{n+1}$ be an elementary tensor in $\cF_n^{j,h}$. Then
\begin{equation}\label{eq:almost}
(sd+ds)(\et) \equiv \sum_{i\in I_\et} \left[ \;
\begin{aligned}
    \et
  & \;+\; \filler{i-1} \tp \lcu{x_ix_{i+1}} \otimes x_ix_{i+1} \tp \filler{n-i} \\
  & \;-\; \filler{i-1} \otimes \lcu{x_i} \otimes x_i x_{i+1} \tp \filler{n-i}
\end{aligned}
\; \right] \mod G_{n,j,h-1}.
\end{equation}
\end{prop}

\FGedit{To be checked as the current proofs in section 6 use cyclic shifts.}

\begin{rem}
When $n+1\in I_\et$, the corresponding term in square brackets should be interpreted as
\begin{equation}\label{eq:different-at-edge}
\et \;+\; (-1)^n \subchain{2}{n} \tp \lcu{x_{n+1}x_1}\tp x_{n+1}x_1  \;+\; (-1)^{n+1} \subchain{2}{n} \tp \lcu{x_{n+1}}\tp x_{n+1} x_1 \,.
\end{equation}
\end{rem}

The proof of Proposition~\ref{p:modified-main} is rather long, and will therefore be deferred to Section~\ref{s:IMGOINGSLIGHTLYMAD}.
Let us assume, for the moment, that the proposition holds; we shall show, under this hypothesis, how Theorem~\ref{t:reduction} can be proved.

For $k=1,\dots,j$, let $P_k = I - k^{-1}(sd+ds)$.
By construction, if $\psi\in \ZC^n(\cA)$ then $\psi- P_k^*\psi= k^{-1}\dif\sigma\psi \in \BC^n(\cA)$, and so applying $P_k^*$ to a cyclic cocycle does not change its cyclic cohomology class.
Proposition~\ref{p:modified-main} suggests that, by repeatedly applying $P_k$ to an elementary tensor in $\cF_n^j$, for varying $k$, one would eventually obtain a linear combination of tensors in $\cF_n^{j-1}$.
To prove that this hope can be realized -- at least, if we work up to cyclic equivalence, see Definition~\ref{dfn:cyc-equiv} -- we must analyze the surviving terms in \eqref{eq:almost} in more detail.
Left-blocks of length one will play a special role and we adopt the following definitions.

\begin{Definition}\label{dfn:block-unit}
A left-block of length one is called a \dt{one-block}.
A one-block $x_k$  in an elementary tensor $\et$ is called a \dt{block-unit} if
$x_k=\lcu{x_k}$ and $x_kx_{k+1}=x_{k+1}$\/.
\end{Definition}

\begin{Remark}\label{rem:not-all-block-units}
Since block-units are left-blocks of length one, there are certainly no more than $j-1$ of them when $j<n+1$. In fact,
if $j=n+1$, then this is still true: for if there were only block-units, the degree of each would lie above that of the following block-unit, and hence the tensor would have a minimal element (and therefore no left-blocks). 
\end{Remark}

Given an elementary tensor $\et=\subchain{1}{n+1}$, let
\[ \sR_\et=\{ i \;:\; \text{ $x_i$ is a one-block but not a block-unit, and $[x_i]\succ [x_{i+1}]$} \}, \]
where we allow $n+1\in \sR_{\et}$. Clearly $\sR_\et$ is a proper subset of the set $I_\et$ of all initial points of left-blocks; it may even be empty.

\begin{lem}\label{l:picky}
Let $\et=\subchain{1}{n+1}$ be an elementary tensor in $\cF_n^j$, and let $i\in I_\et$. Then precisely one of the following four cases can occur:
\begin{YCnum}
\item\label{item:cancel}
 $x_i$ is not a one-block in $\et$, in which case,
\begin{equation}\label{eq:cancel}
    \filler{i-1} \tp \lcu{x_ix_{i+1}} \otimes x_ix_{i+1} \tp \filler{n-i}
    \quad = \quad
    \filler{i-1} \otimes \lcu{x_i} \otimes x_i x_{i+1} \tp \filler{n-i}\,;
\end{equation}
\item\label{item:drop}
 $x_i$ is a one-block and $[x_i]\not\succeq [x_{i+1}]$, in which case, the tensor
\begin{equation}\label{eq:drop}
\filler{i-1}\otimes\lcu{x_i}\otimes x_ix_{i+1}\tp\filler{n-i}
\end{equation}
either has fewer left-blocks, or lower height (and the same number of left-blocks), than~$\et$\/.
\item\label{item:BU}
 $x_i$ is a block-unit, in which case $\filler{i-1}\otimes\lcu{x_i}\otimes x_ix_{i+1}\tp\filler{n-i}=\et$\/;
\item\label{item:ELSE} $i\in \sR_\et$.
\end{YCnum}
In cases \ref{item:drop}, \ref{item:BU} and \ref{item:ELSE}, the tensor
$\filler{i-1} \tp \lcu{x_ix_{i+1}} \otimes x_ix_{i+1} \tp \filler{n-i}$ has fewer left-blocks than~$\et$.
\end{lem}

\begin{proof}
If $x_i$ is not a one-block, then since $i$ is initial we must have $[x_i]\preceq[x_{i+1}]$. Equation~\eqref{eq:cancel} then follows from the left-coherent property of the function $\lcu{\cdot}$, as described in Section~\ref{temporary-marker}.

If $x_i$ is a one-block, we split into two cases. The first is when $[x_i]$ does not lie above $[x_{i+1}]$, i.e.~case \ref{item:drop} of the lemma. In this case, $x_ix_{i+1}$ has strictly smaller degree than $x_{i+1}$, and so \eqref{eq:drop} has
height at most
\[ \high(\et) - \high_{L(\et)}([x_{i+1}]) + \high_{L(\et)}([x_ix_{i+1}]) < \high(\et) \]
as claimed.
The second case is when $[x_i]\succeq [x_{i+1}]$ (note that, since $x_i$ is assumed here to be a one-block, it then has to lie \emph{strictly} above $x_{i+1}$). There are now two subcases: either $x_i$ is a block-unit, in which case the claim in \ref{item:BU} follows immediately from the definition of a block-unit (and the fact that $\lcu{\lcu{\cdot}}=\lcu{\cdot}$); or else it is not, in which case $i$ is by definition a member of $\sR_\et$, so that we are in case~\ref{item:ELSE}.

Finally: if we are not in case \ref{item:cancel}, i.e.~if $x_i$ is a one-block, then
 $\filler{i-1} \tp \lcu{x_ix_{i+1}} \otimes x_ix_{i+1} \tp \filler{n-i}$
 clearly has fewer left-blocks than~$\et$ (more precisely, the left-blocks which started in position $i$ and position $i+1$ have merged).
\end{proof}

The previous lemma motivates the following notation.
Define a map $\err$ on elementary tensors by
\begin{equation}\label{eq:error-terms}
\err(\et) =  \sum_{i\in \sR_{\et}}\; \filler{i-1} \tp \lcu{x_i} \tp x_ix_{i+1} \tp \filler{n-i}
\end{equation}
where if $\sR_{\et}=\emptyset$ we put $\err(\et)\defeq 0$; and extend $\err$ by linearity and continuity to a bounded linear map on $\Ho{C}{n}(\cA)$.
It is easily checked from the definitions in \eqref{eq:define_spaces} that $\err$ maps $G_{n,j,h}$ into itself, and hence maps $H_{n,j,h}$ into itself.

\begin{cor}\label{c:now-we-need-cyclic}
If $\et=\subchain{1}{n+1}$ is an elementary tensor with height $h$, and with $j$ left-blocks, exactly $r$ of which are block-units, then
\begin{equation}\label{eq:Brel} 
P_k(\et) \equiv \left(1-\frac{j-r}{k}\right)\et + \frac{1}{k}\err(\et) \ \mod H_{n,j,h-1}\,.
\end{equation}
\end{cor}

\begin{proof}
Fix $i\in I_\et$ and consider the corresponding terms enclosed by square brackets on the right-hand side of \eqref{eq:almost} (or, if $i=n+1$, the terms in \eqref{eq:different-at-edge}).

If $x_i$ is a block-unit and $1\leq i \leq n$, then by Lemma~\ref{l:picky}\ref{item:BU} the first and third of these terms cancel out, while the middle term is equal to $\filler{i-1}\tp\lcu{x_ix_{i+1}}\tp x_ix_{i+1}\tp\filler{n-i}$ and so lies in $\cF_n^{j-1}$. If $i=n+1$ and $x_{n+1}$ is a block-unit, we have to consider \eqref{eq:different-at-edge}: there, the first and third terms will now cancel out \emph{modulo cyclic equivalence}; while the middle term is cyclically equivalent to $x_{n+1}x_1\tp\filler{n-1}\tp \lcu{x_{n+1}x_1}$ and so lies in $\cF_n^{j-1}$ as before.

If $x_i$ is not a block-unit, then we get $\et$, together with two other terms. These two will cancel if $x_i$ is not a one-block (Lemma~\ref{l:picky}\ref{item:cancel}); while if $x_i$ is a one-block that does not lie above its successor, these terms will either have fewer left-blocks or lower height than $\et$ (Lemma~\ref{l:picky}\ref{item:drop}). That leaves only the case where $i\in R_\et$, when one of the terms will have fewer left-blocks and the other will form part of $\err(\et)$.

Summing up over all $i\in I_\et$, we obtain from \eqref{eq:almost}
\[ \begin{aligned}
(ds+sd)(\et)
 & \equiv \left( \sum_{\text{$i\in I_{\et}$, $x_i$ not a block-unit}} \et\right)
  \; - \; \left(\sum_{i\in R_\et} \filler{i-1}\tp\lcu{x_i}\tp x_ix_{i+1} \tp \filler{n-i}\right) \mod H_{n,j,h-1} \\
\end{aligned} \]
The first sum in brackets is equal to $(j-r)\et$; the second is equal to $\err(\et)$; and now rearranging gives us the desired identity.
\end{proof}


At this point, note that the degree of a left-block and the height of a tensor, which both play a pivotal role in our analysis, depend only on the degrees of terms in an elementary tensor, i.e.~which elements of the structure semilattice~$L$ index these terms. This motivates the following definition.
\begin{dfn}
Given an elementary tensor $\et=\subchain{1}{n+1}$, the \dt{shape} of $\et$ is the tensor $[x_1]\tp \cdots\tp[x_{n+1}] \in \ell^1(L^{n+1})$.
\end{dfn}

Clearly, the number of left-blocks, the location of initial points of left-blocks, and the height, are each dependent only on the shape of a tensor. It is also clear that
if $i\in \sR_\et$\/, then $\filler{i-1}\tp \lcu{x_i}\tp x_ix_{i+1} \tp \filler{n-i}$ has the same shape as $\et$. Consequently, each term in $P_k(\et)$ either has the same shape as $\et$, or else has lower height, or else has fewer left-blocks.

\medskip
\para{Descent of an elementary tensor}
Given an elementary tensor $\et=\subchain{1}{n+1}$ without minimal element, define a \dt{descending block} in $\et$ to be
 a cyclic subtensor $x_k\tp\cdots \tp x_l$ with the property that
 $[x_k] \succ [x_{k+1}]\succ \dots \succ [x_l]$, while $[x_{k-1}]\not\succ[x_k]$ and $[x_l]\not\succ[x_{i+1}]$.
Since the entries of $\et$ can strictly decrease at most $n$ times,
a descending block in $\et=\subchain{1}{n+1}$ can have length at most $n$, and so has a well-defined first element and last element.
In particular, for each $i$ we can define the \dt{descent of $x_i$ in~$\et$} to be $l-i$,
where $x_l$ is the last element in the unique descending block that contains $x_i$.
(This is interpreted cyclically, so that if $x_n\tp x_{n+1}\tp x_1$ is a descending block, then the descent of $x_{n+1}$ is~$1$.)
We denote the descent of $x_i$ in $\et$ by $\desc_i(\et)$, and now define the \dt{descent of $\et$} to be
\[ \desc(\et) \defeq \sum_{i \in \sR_\et} \desc_i(\et). \]
Since $\desc_i(\et)\leq n-1$ for all $i$ and $\abs{\sR_\et}\leq j-1$, there is a crude upper bound $\desc(\et)\leq (j-1)(n-1)$. Moreover, since $\desc{x_i} \geq 1$ for each $i\in\sR_\et$ -- recall that if $i\in \sR_\et$ then $x_i$ lies \emph{strictly} above its successor -- there is a lower bound $\desc(\et) \geq \abs{\sR_\et}$.

The idea behind the next lemma is that, given $\et$ with $\sR_\et$ non-empty, each term in $\err(\et)$ either has one more block-unit, or else has one of the block-units shifted one place to the left; since each such term has the same shape as $\et$, this process must terminate after a finite number of steps.

\begin{lem}\label{l:tombant}
Let $\et=\subchain{1}{n+1}$ be a tensor without minimal element, such that $R_\et$ is non-empty, and let $i\in \sR_\et$. Then
\[ \desc\left( \filler{i-1}\tp\lcu{x_i}\tp x_ix_{i+1} \tp \filler{n-i}\right) < \desc(\et). \]
\end{lem}

\begin{proof}
To avoid potential concern over degenerate cases, we start by observing that since $\et$ has no minimal element but $\sR_\et$ is non-empty, we must have $n\geq 2$. Next, since the definition of descent is cyclically invariant, we may as well cycle our tensor so that $2\leq i \leq n$ (this just simplifies some of the notational book-keeping).

Put $\yy=\filler{i-1}\tp\lcu{x_i}\tp x_i x_{i+1} \tp\filler{n-i}$\/.
Note that $\yy$ coincides with $\et$ in the first $i-1$ and last $n-i$ entries, and $[y_k]=[x_k]$ for all $k\in\{1,\dots,n+1\}$.
Given $r\in\{1,\dots,i-1\}\cup\{i+2,\dots, n+1\}$, it follows that $y_r$ is a one-block and non-block-unit lying strictly above its successor in $\yy$, if and only $x_r$ is a one-block and non-block-unit lying strictly above its successor in $\et$; moreover, if this is the case then the descent of $y_r$ is equal to that of $x_r$. It follows from the definition of descent that
\[
  \desc(\et) - \sum_{r\in \sR_\et \cap \{i,i+1\}} \desc_r(\et)
= \desc(\yy) - \sum_{r\in \sR_\yy \cap \{i,i+1\}} \desc_r(\yy).
 \]
Now by hypothesis $i\in \sR_\et$; and since $y_i$ is a block-unit in $\yy$, we have $i\notin \sR_\yy$. Moreover, since $[x_i]\succeq [x_{i+1}]$, it is clear that the descent of $x_i$ in $\et$ is strictly greater than the descent of $y_{i+1}=x_ix_{i+1}$ in $\yy$. Therefore
\[ \sum_{r\in \sR_\et \cap \{i,i+1\}} \desc_r(\et) \geq \desc_i(\et) > \desc_{i+1}(\yy) \geq \sum_{r\in \sR_\yy \cap \{i,i+1\}} \desc_r(\yy) \]
and hence $\desc(\et) > \desc(\yy)$ as claimed.
\end{proof}

\begin{cor}
Let $\et\in\cF_n^{j,h}$. Then $(P_j\dotsb P_1)(\et)$ is
 congruent $\mod H_{n,j,h-1}$ to a linear combination of elementary tensors which have smaller descent than~$\et$.
In particular, if $N\geq j(n-1)n^2$,
then
$(P_j\dotsb P_1)^N(\et)$ is cyclically equivalent to a tensor in $\lin \cF_n^{j-1}$.
\end{cor}

\begin{proof}
Let $h\defeq \high(\et)$.
First, note that the operators $P_1,\dots, P_j$ are pairwise commuting (as they are just linear combinations of $I$ and $sd+ds$). Note also that by Corollary~\ref{c:now-we-need-cyclic} and the remarks preceding it, each $P_i$ maps $H_{n,j,h-1}$ to itself.

Now, if $r$ is the number of block-units in $\et$, let
\[ Q_{j-r} =P_j\dotsb P_{j-r+1}P_{j-r-1}\dotsb P_1\,.\]
We note that $Q_{j-r}$ maps $H_{n,j,h-1}$ to itself. Hence,
recalling that $0\leq r\leq j-1$, it follows from Corollary~\ref{c:now-we-need-cyclic} that
\begin{equation}\label{eq:BUTTERMAN}
P_j\dots P_1(\et) = Q_{j-r} P_{j-r}(\et) \equiv Q_{j-r}(\err(\et)) \quad\mod H_{n,j,h-1}.
\end{equation}

Let $\yy$ be an elementary tensor in $G_{n,j,h-1}$.
For arbitrary $k$, the identity \eqref{eq:Brel} also implies that the tensor $P_k(\yy)$ is cyclically equivalent to a linear combination of a term in $G_{n,j,h-1}$, some scalar multiple of $\yy$, and some scalar multiple of $\err(\yy)$\/; in particular, $\mod H_{n,j,h-1}$, $P_k(\yy)$ is a linear combination of terms whose descent does not exceed $\desc(\yy)$.
(This uses Lemma~\ref{l:tombant} applied to~$\yy$.)

Since $Q_{j-r}$ is a product of various $P_k$, the same is true of $Q_{j-r}(\yy)$; and so the descent of \emph{each term} in $Q_{j-r}(\err(\et))$ is bounded above by $\desc(\err(\et))$, which is in turn strictly less than $\desc(\et)$.
Combining this with \eqref{eq:BUTTERMAN}, we see that, $\mod H_{n,j,h-1}$, the tensor $P_j\dotsb P_1(\et)$ is  a linear combination of terms with descent \emph{strictly} less than $\desc(\et)$.

Now let $P=(P_j\dots P_1)^{j(n-1)}$. Since $\desc(\et)\leq (j-1)(n-1) \leq j(n-1)-1$, the previous paragraph implies that
\[ P(\et)\equiv 0 \; \mod H_{n,j,h-1}. \]
That is, $P(\et)$ is cyclically equivalent to a linear combination of terms that have at most $j-1$ left-blocks, together with terms that have
height strictly less than $\high(\et)$.
%
Finally, we can iterate again, using the fact that $n+1\leq \high(\et) \leq n(n+1)$, to deduce that if we apply $P$ to $\et$ at least $n(n+1)-(n+1)+1 = n^2$ times, then the resulting tensor will be cyclically equivalent to one in $\lin\cF_n^{j-1}$. This concludes the proof.
\end{proof}

To prove Theorem \ref{t:reduction}, we take $N=j(n-1)n^2$: if $\psi\in\SC{n}{\cA}$ vanishes on all tensors in $\cF_n^{j-1}$, the cochain
$\psi_1\defeq \left[(P_j\dots P_1)^N\right]^*(\psi)$
vanishes on all tensors in $\cF_n^j$\/, by the previous corollary. By our earlier remarks, $\psi_1$ is in the same cyclic cohomology class as $\psi$, and we have proved Theorem~\ref{t:reduction}, provided we take for granted the proof of Proposition~\ref{p:modified-main}.
\hfill$\Box$
\end{subsection}
\end{section}

\begin{section}{The proof of Proposition~\ref{p:modified-main}}\label{s:IMGOINGSLIGHTLYMAD}
Throughout, $\et$ denotes a fixed elementary tensor $\subchain{1}{n+1}$ which has exactly $j$ left-blocks.

As in earlier sections, it will be useful to regard the boundary operator $d$ as an alternating sum of \dt{face maps}.

\begin{Definition}[Face maps on $\Ho{C}{*}(\cA)$]
For $i=0,\dots, n$, define the \dt{face maps} from $\Ho{C}{n}(\cA)$ to $\Ho{C}{n-1}(\cA)$ by
\[ \begin{aligned}
\face_0(\chain) & = x_2\tp \cdots \tp x_n \tp x_{n+1}x_1 \\
\face_i(\chain) & = \filler{i-1}\tp x_ix_{i+1}\tp \filler{n-i} \qquad\text{ for $1\leq i \leq n$.}
\end{aligned}  \]
\end{Definition}

\para{An easy but key observation}
If $a,b,c\in S$ and $[a]\preceq [bc]$, then $[a]\preceq [b]$ and $[a]\preceq [c]$. Thus, if $\et = \lbchain$, where each $w_l$ is a left-block, and if $x_r$ and $x_{r+1}$ are contained in the same left-block $w_l$, then $\face_r(\et)$ also has exactly $j$ left-blocks, and has the form
\[ \face_r(\et) = \lbsubchain{1}{l-1} \tp w_l'\tp \lbsubchain{l+1}{j}\,, \]
where the only new left-block, $w_l'$, is just $w_l$ with $x_r$ and $x_{r+1}$ multiplied together. The important point is that $w_l'$ does not become part of a larger left-block.

\medskip
If, on the other hand, the face map does the product of the end of one left-block with the start of the next left-block, then the resulting tensor might have $j$ left-blocks, but might have fewer. The following example illustrates some possibilities.

\begin{Example}[An illustration of complications]
Let $S$ be the free semilattice on $4$ generators, labelled as $g_1$, $g_2$, $g_3$ and~$g_4$. Consider
\[ \et = \overbrace{g_1g_2\tp g_2}\tp \overbrace{g_1g_2g_3\tp g_1g_2} \tp \overbrace{g_3g_4} \tp \overbrace{g_1} \tp \overbrace{g_1g_3 \tp g_3} \in \Ho{C}{7}(\cA) \]
which consists of $5$ left-blocks as indicated (so that $I_\et=\{1,3,5,6,7\}$\}). Then
\[ \face_4(\et) = g_1g_2\tp g_2\tp g_1g_2g_3 \tp g_1g_2g_3g_4 \tp g_4 \tp g_1g_2 \tp g_1g_3 \tp g_3 \in \Ho{C}{6}(\cA) \]
contains a \emph{minimal element}, so that all entries lie in the same left-block. Note that the set $I_4$ of initial points in $\face_4(\et)$ is just~$\{4\}$.

For sake of comparison, note that
\[ \begin{aligned}
  \face_5(\et) & = \overbrace{g_1g_2\tp g_2}\tp \overbrace{g_1g_2g_3\tp g_1g_2} \tp \overbrace{g_1g_3g_4 \tp g_1g_3 \tp g_3} & \quad I_5=\{1,3,5\} \\
  \face_6(\et) & = \overbrace{g_1g_2\tp g_2}\tp \overbrace{g_1g_2g_3\tp g_1g_2} \tp \overbrace{g_3g_4} \tp \overbrace{g_1g_3 \tp g_3} & \quad I_6=\{1,3,5,6\} \\
  \face_2(\et) & = \overbrace{g_1g_2} \tp \overbrace{g_1g_2g_3\tp g_1g_2} \tp \overbrace{g_3g_4} \tp \overbrace{g_1} \tp \overbrace{g_1g_3 \tp g_3}
	& \quad I_2 =\{1,2,4,5,6\}
\end{aligned} \]
\end{Example}

With this warning example in mind, we start on the proof.
Define indexing sets $\SD\subseteq \{1,\dots, n\}\times \{0,\dots,n\}$ and $\DS\subseteq \{0,\dots,n+1\}\times\{1,\dots, n+1\}$ by
\[ \begin{aligned}
\SD & = \{ (i,p) \st 0\leq p \leq n \text{ and } i\in I_{\face_p(\et)} \}
    & = \coprod_{0\leq p \leq n}  I_{\face_p(\et)}\times\{p\} \\
\DS & = \{(r,k) \st 0\leq r \leq n+1 \text{ and } k\in I_{\et} \}
    & = \{0,\dots,n+1\}\times I_{\et}\,.
\end{aligned} \]
Then
\[ (sd+ds)(\et) = \sum_{(i,p)\in\SD} (-1)^p s_i \face_p(\et) + \sum_{(r,k)\in\DS} (-1)^r \face_r s_k(\et), \]
and the first task in evaluating this tensor is to show that most terms on the right-hand side either cancel pairwise, or have fewer than $j$ left-blocks, or have lower height than~$\et$. Much of this takes place in greater generality, without using the properties of the left-coherent units that are inserted.

\begin{lem}\label{l:generic-identities}\
\begin{subequations}
\begin{enumerate}
\renewcommand{\labelenumi}{{\rm(\Alph{enumi})}}
\renewcommand{\theenumi}{{\rm(\Alph{enumi})}}
\item\label{item:A} Let $1\leq i < p \leq n$. Then
\begin{equation}\label{eq:insert-pinch}
\begin{aligned}
  s_i \face_p (\et)
 & = (-1)^i \filler{i-1} \tp \lcu{x_i}\tp \subchain{i}{p-1} \tp x_px_{p+1}\tp \filler{n-p-1}
 & =  \face_{p+1} s_i(\et)
\end{aligned}
\end{equation}
\item\label{item:B} Let $1\leq p < i\leq n$. Then
\begin{equation}\label{eq:pinch-insert}
\begin{aligned}
  s_i\face_p(\et)
 & = (-1)^i \filler{p-1} \tp x_p\subchain{p+1}{i} \tp \lcu{x_{i+1}}\tp x_{i+1} \tp \filler{n-i}
 & = - \face_p s_{i+1}(\et)
\end{aligned}
\end{equation}
\item\label{item:C} Let $1\leq i \leq n-1$. Then
\begin{equation}\label{eq:edges}
s_i\face_0(\et)
 = (-1)^i \subchain{2}{i} \tp \lcu{x_{i+1}}\tp x_{i+1} \tp \filler{n-i}\tp x_{n+1}x_1
 = -\face_0 s_{i+1}(\et)
\end{equation}
\end{enumerate}
\end{subequations}
\end{lem}

\begin{proof}
This is a direct computation. We omit the details: see Figure~\ref{fig:doodles} for a diagram which illustrates how this works in Cases \ref{item:A} and~\ref{item:B}.

\begin{figure}
\caption{Individual pairs of cancelling terms}
\label{fig:doodles}
%
\newcommand{\dlin}{\line(0,2){2}}
\newcommand{\llin}{\line(-3,4){1.5}}
\newcommand{\rlin}{\line(3,4){1.5}}
\newcommand{\hlin}{\line(1,0){1.5}}
\setlength{\unitlength}{1em}
\begin{picture}(40,10)
\put(0.8,8.5){$\scriptstyle 1$}
\put(2.3,8.5){$\scriptstyle 2$}
\put(3.8,8.5){$\scriptstyle 3$}
\put(5.3,8.5){$\scriptstyle 4$}
\put(6.8,8.5){$\scriptstyle 5$}
\put(1.0,6){\dlin}
\put(2.5,6){\dlin}
\put(4.0,6){\dlin}
\put(5.5,6){\dlin}
\put(5.5,6){\hlin}
\put(7.0,6){\llin}
\put(8.5,6){\llin}
\put(0.8,5.2){$\scriptstyle 1$}
\put(2.3,5.2){$\scriptstyle 2$}
\put(3.8,5.2){$\scriptstyle 3$}
\put(5.3,5.2){$\scriptstyle 4$}
\put(6.8,5.2){$\scriptstyle 4$}
\put(8.3,5.2){$\scriptstyle 5$}
\put(1.0,3){\dlin}
\put(2.5,3){\dlin}
\put(2.5,3){\rlin}
\put(2.5,5){\hlin}
\put(4.0,3){\rlin}
\put(5.5,3){\rlin}
\put(7.0,3){\rlin}
\put(0.8,2){$\scriptstyle 1$}
\put(2.3,2){$\scriptstyle 23$}
\put(3.8,2){$\scriptstyle 4$}
\put(5.3,2){$\scriptstyle 4$}
\put(6.8,2){$\scriptstyle 5$}
\put(3.0,0){$\face_2s_4$}
\put(8.5,0){$=$}
%
%
\put(10.8,8.5){$\scriptstyle 1$}
\put(12.3,8.5){$\scriptstyle 2$}
\put(13.8,8.5){$\scriptstyle 3$}
\put(15.3,8.5){$\scriptstyle 4$}
\put(16.8,8.5){$\scriptstyle 5$}
\put(11.0,6){\dlin}
\put(12.5,6){\dlin}
\put(12.5,6){\rlin}
\put(12.5,8){\hlin}
\put(14.0,6){\rlin}
\put(15.5,6){\rlin}
\put(10.8,5.2){$\scriptstyle 1$}
\put(12.3,5.2){$\scriptstyle 23$}
\put(13.8,5.2){$\scriptstyle 4$}
\put(15.3,5.2){$\scriptstyle 5$}
\put(11.0,3){\dlin}
\put(12.5,3){\dlin}
\put(14.0,3){\dlin}
\put(14.0,3){\hlin}
\put(15.5,3){\llin}
\put(17.0,3){\llin}
\put(10.8,2){$\scriptstyle 1$}
\put(12.3,2){$\scriptstyle 23$}
\put(13.8,2){$\scriptstyle 4$}
\put(15.3,2){$\scriptstyle 4$}
\put(16.8,2){$\scriptstyle 5$}
\put(12,0){$-s_3\face_2$}
%
%
%
\put(23.8,8.5){$\scriptstyle 1$}
\put(25.3,8.5){$\scriptstyle 2$}
\put(26.8,8.5){$\scriptstyle 3$}
\put(28.3,8.5){$\scriptstyle 4$}
\put(29.8,8.5){$\scriptstyle 5$}
\put(24.0,6){\dlin}
\put(25.5,6){\dlin}
\put(25.5,6){\hlin}
\put(27.0,6){\llin}
\put(28.5,6){\llin}
\put(30.0,6){\llin}
\put(31.5,6){\llin}
\put(23.8,5.2){$\scriptstyle 1$}
\put(25.3,5.2){$\scriptstyle 2$}
\put(26.8,5.2){$\scriptstyle 2$}
\put(28.3,5.2){$\scriptstyle 3$}
\put(29.8,5.2){$\scriptstyle 4$}
\put(31.3,5.2){$\scriptstyle 5$}
\put(24.0,3){\dlin}
\put(25.5,3){\dlin}
\put(27.0,3){\dlin}
\put(28.5,3){\dlin}
\put(28.5,5){\hlin}
\put(28.5,3){\rlin}
\put(30.0,3){\rlin}
\put(23.8,2){$\scriptstyle 1$}
\put(25.3,2){$\scriptstyle 2$}
\put(26.8,2){$\scriptstyle 2$}
\put(28.3,2){$\scriptstyle 34$}
\put(29.8,2){$\scriptstyle 5$}
\put(26,0){$\face_4s_2$}
\put(31.5,0){$=$}
\put(33.8,8.5){$\scriptstyle 1$}
\put(35.3,8.5){$\scriptstyle 2$}
\put(36.8,8.5){$\scriptstyle 3$}
\put(38.3,8.5){$\scriptstyle 4$}
\put(39.8,8.5){$\scriptstyle 5$}
\put(34.0,6){\dlin}
\put(35.5,6){\dlin}
\put(37.0,6){\dlin}
\put(37.0,8){\hlin}
\put(37.0,6){\rlin}
\put(38.5,6){\rlin}
\put(33.8,5.2){$\scriptstyle 1$}
\put(35.3,5.2){$\scriptstyle 2$}
\put(36.8,5.2){$\scriptstyle 34$}
\put(38.3,5.2){$\scriptstyle 5$}
\put(34.0,3){\dlin}
\put(35.5,3){\dlin}
\put(35.5,3){\hlin}
\put(37.0,3){\llin}
\put(38.5,3){\llin}
\put(40.0,3){\llin}
\put(33.8,2){$\scriptstyle 1$}
\put(35.3,2){$\scriptstyle 2$}
\put(36.8,2){$\scriptstyle 2$}
\put(38.3,2){$\scriptstyle 34$}
\put(39.8,2){$\scriptstyle 5$}
\put(36,0){$s_2\face_3$}
\end{picture}
\end{figure}
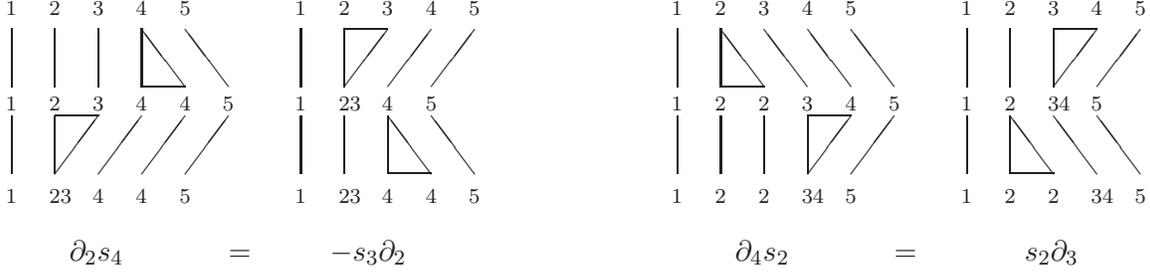
\FGedit{I don't see a cyclic shift in C, and have removed the part about extra care.}
\end{proof}

We have to keep track of which terms in a corresponding pair, as in Lemma~\ref{l:generic-identities}, actually occur when we expand out $(sd+ds)(\et)$.
More notation will be useful. Let
\[ \begin{aligned}
 A & = \{(i,p) \st 1\leq i < p \leq n\},
   & \quad A'& = \{(r,k) \st 1\leq k < r-1 \leq n \}, \\
 B & = \{(i,p) \st 1\leq p < i\leq n \},
   & \quad B'& = \{(r,k) \st 1\leq r < k-1 \leq n \}, \\
 C & = \{(i,0) \st 1\leq i \leq n-1\},
   & \quad C'& = \{(0,k) \st 2\leq k \leq n \},
\end{aligned}  \]
so that $A$, $B$ and $C$ are pairwise disjoint subsets of $\{1,\dots,n\}\times\{0,\dots,n\}$ and $A'$, $B'$ and $C'$ are pairwise disjoint subsets of $\{0,\dots,n+1\}\times\{1,\dots,n+1\}$.
Now put
\[ \begin{aligned}
 \DS_j & \defeq \{ (r,k) \in \DS \st \face_rs_k(\et) \text{ has exactly $j$ left-blocks} \}, \\
 \SD_j & \defeq \{ (i,p) \in \SD \st \face_p(\et)  \text{ has exactly $j$ left-blocks}  \},
\end{aligned} \]
and let
$\DS_j^* \defeq \DS_j\cap (A' \sqcup B' \sqcup C')$
and
$\SD_j^* \defeq \SD_j\cap (A\sqcup B\sqcup C)$.

\medskip
It turns out that the obvious maps $A\leftrightarrow A'$, $B\leftrightarrow B'$ and $C\leftrightarrow C'$ restrict to give a bijection between $\DS^*_j$ and $\SD^*_j$ (which shows that most terms in $\sum_{\DS_j} +\sum_{\SD_j}$ cancel pairwise). To do this precisely, we have a lemma.

\begin{lem}\label{l:when-pinching-reduces}
Let $m\geq 1$, and let $\yy= y_1\tp\cdots \tp y_{m+1}$ be an elementary tensor with $j$ left-blocks.
Let $1 \leq p \leq m$. Then $\face_p(\yy)\in\cF_n^{j-1}$ if and only if one (or both) of the following holds: either {\rm (a)}~$y_p$ is a $1$-block in $\yy$; or {\rm (b)}~$[y_py_{p+1}]\preceq [w]$, where $w$ is the left-block immediately following the one which contains~$y_{p+1}$\/.
An analogous result holds for $\face_0(\yy)$, provided that condition~{\rm (a)} is interpreted as ``$y_{m+1}$ is a $1$-block in $\yy$'', and condition~{\rm (b)} as ``\/$[y_{m+1}y_1]\preceq [w]$ \dots''.
\end{lem}

\begin{proof}
We first note that the case $p=0$ is not really distinct from the cases $1\leq p \leq m$, once we interpret `position $0$' in a tensor of length $m+1$ as being position $m+1$.
Next, we may assume without loss of generality, that $1\in I_\yy$. (For if not, then by applying a suitable power of $\cyc$ we obtain a tensor $\yy'$ in which 1 is an initial point, and work with $\yy'$ instead.)
Let $\yy = \lbchain$ be the decomposition of $\yy$ into its constituent left-blocks. Let~$w_k$ be the left-block which contains~$y_{p+1}$.

If (a) holds, then $w_{k-1}$ just consists of the single element $y_p$, and
\[ \face_p(\yy) = \lbsubchain{1}{k-2} \tp y_p\cdot \lbsubchain{k}{j}\,; \]
thus two left-blocks have been merged together, and there are now at most $j-1$ of them.
If~(b) holds, then every element of $w$ and every element of $w_k$ lies above $[y_py_{p+1}]$, so that in $\face_p(\yy)$ these two left-blocks are merged into a single one; thus once again, the number of left-blocks has decreased.

Conversely suppose that $\face_p(\yy)$ has fewer than $j$ left-blocks, and suppose (b) does not hold. Then the left-blocks $w_1,\dots, w_{k-2}$ and $w_{k+1}, \dots, w_j$ remain left-blocks in $\face_p(\yy)$. Therefore $w_{k-1}\cdot w_k$ must form a single left-block. If $y_r$ denotes the initial element of $w_{k-1}$, and $r<p$, then this implies that $[y_r]\preceq [y_py_{p+1}]\preceq [y_{p+1}]$ and this contradicts the fact that $w_{k-1}$ and $w_k$ are disjoint left-blocks. The only remaining possibility is that $w_{k-1}$ is a $1$-block, with $y_p$ as its sole element, and so (a) holds.
\end{proof}

In view of condition (b) in this lemma, we say that the tensor $\yy$ has a \dt{dead spot at~$p+1$}, for $0\leq p \leq m$, if $[x_px_{p+1}]\preceq [w]$, where $w$ is the left-block immediately following the one which contains~$x_{p+1}$\/.
Once again, this definition should be interpreted cyclically, so that having a dead spot at $1$ means that $[x_{m+1}x_1]\preceq [w]$, etc.

\begin{prop}
Define $\phi:\DS_j^* \to \{1,\dots,n\}\times\{1,\dots,n\}$ by
\begin{equation}
\phi(r,k) = \left\{
\begin{aligned}
(k-1,r) & \quad\text{ if $(r,k)\in B'\cup C'$} \\
(k,r-1) & \quad\text{ if $(r,k)\in A'$}.
\end{aligned}
\right.
\end{equation}
Then $\phi$ maps $\SD_j^*$ bijectively onto $\DS_j^*$.
Consequently,
\begin{equation}\label{eq:kill-off-diag}
   \sum_{(r,k)\in \DS_j^*} (-1)^r \face_r s_k(\et)
 + \sum_{(i,p)\in \SD_j^*} (-1)^p s_i \face_r(\et) = 0 \,.
\end{equation}
\end{prop}

\begin{proof}
Start by noting that $\phi$ is the restriction of obvious bijections from $A'$, $B'$ and $C'$ to $A$, $B$ and $C$ respectively.

Moreover: the identities in Lemma~\ref{l:generic-identities} show that if $(r,k)\in\DS_j^*$, then $\phi(r,j)\in \SD_j^*$\/. (The point is that if, say, $1\leq k\leq r-2$ and $\face_rs_k(\et)$ has $j$ left-blocks, then the identity \eqref{eq:pinch-insert} shows that $s_{k-1}\face_r(\et)$ has $j$ left-blocks, so $\face_r(\et)$ must have $j$ left-blocks.) Thus, $\ran\phi \subseteq \SD_j^*$.

To show the converse inclusion: let $(i,p)\in\SD_j^*$. Then by Lemma~\ref{l:when-pinching-reduces} (with $m=n$), $x_p$ is not a $1$-block in $\et$ and $p+1$ is not a dead spot in $\et$. (If $p=0$ this means $x_{n+1}$ is not a $1$-block, etc.)
Therefore, by the other direction of Lemma~\ref{l:when-pinching-reduces} (with $m=n+1$):
\begin{itemize}
\item[--] if $1\leq i \leq p-1$, and we consider $s_i(\et)$, then $x_p$ (occuring in position $p+1$) is not a $1$-block in $s_i(\et)$ and $p+2$ is not a dead spot in $s_i(\et)$, so that $(i,p+1)\in\DS^*_j$\/;
\item[--] if $2\leq p+1\leq i \leq n$, and we consider $s_{i+1}(\et)$, then $x_p$ (occuring in position $p$) is not a $1$-block in $s_{i+1}(\et)$ and $p+1$ is not a dead spot in $s_{i+1}(\et)$, so that $(i+1,p)\in\DS^*_j$\/.
\item[--] if $p=0$ and $1\leq i \leq n$, and we consider $s_{i+1}(\et)$, then $x_{n+1}$ (occuring in position $n+2$) is not a $1$-block in $s_{i+1}(\et)$ and $1$ is not a dead spot in $s_{i+1}(\et)$, so that $(i+1,0)\in\DS^*_j$\/.
\end{itemize}
In each case, $(i,p)\in\ran\phi$ as required.
\end{proof}

\medskip
We now continue with the proof of Proposition~\ref{p:modified-main}.
It follows from \eqref{eq:kill-off-diag} that,
\begin{equation}\label{eq:interim}
(sd+ds)(\et)
 \equiv \sum_{(r,k)\in \DS_j\setminus \DS_j^*} (-1)^r \face_r s_k(\et)
   + \sum_{(i,p)\in \SD_j\setminus \SD_j^*} (-1)^p s_i\face_p(\et) \; \mod (\lin \cF_n^{j-1}).
\end{equation}
Expanding out the terms on the right-hand side gives
\begin{subequations}
\begin{eqnarray}
   & \sum_{1\leq k\leq n+1\st (k-1,k)\in \DS_j} &(-1)^{k-1} \face_{k-1}s_k(\et)
	\label{eq:ds1}  \\
 + & \sum_{1\leq k \leq n+1\st (k,k)\in \DS_j}  &(-1)^k\face_ks_k(\et)
	\label{eq:ds2}  \\
 + & \sum_{1\leq k \leq n \st (k+1,k)\in \DS_j} &(-1)^k \face_{k+1} s_k(\et) \\
 + & \sum_{1\leq i\leq n \st (i,i)\in\SD_j}     &(-1)^i s_i\face_i(\et)
	\label{eq:sd}  \\
 + & R_{\DS}(\et) + R_{\SD}(\et),
\label{eq:edge-cases}
\end{eqnarray}
\end{subequations}
where the terms $R_{\DS}(\et)$ and $R_{\SD}(\et)$ are defined by
\[ \begin{aligned}
R_{\DS}(\et) & =
	\text{ $\face_0 s_{n+1}(\et)$ if $(0,n+1)\in\DS_j$}
	& \text{ and	$0$ otherwise,} \\
R_{\SD}(\et) & =
	\text{ $s_n\face_0(\et)$ if $(n,0)\in\SD_j$}
	& \text{ and $0$ otherwise.}
\end{aligned} \]

\begin{lem}\label{l:Donna}\
\begin{YCnum}
\item\label{item:ds-lower} If $0\leq k \leq n$ and $k+1\in I_\et$\/, then $\face_ks_{k+1}(\et)$ has strictly lower height than~$\et$.
\item\label{item:sd-lower} If $1\leq k \leq n$ and $k+1\in I_\et$, then $s_k\face_k(\et)$ has strictly lower height than~$\et$\/. If $1\in I_\et$, then $s_n\face_0(\et)$ has strictly lower height than~$\et$.
\end{YCnum}
\end{lem}

\begin{proof}
First suppose that $1\leq k \leq n$. Then the corresponding terms in \ref{item:ds-lower} and \ref{item:sd-lower} expand out to be
\[
 \aa = (-1)^{k+1} \filler{k-1}\tp x_k\lcu{x_{k+1}}\tp x_{k+1} \tp \filler{n-k}
 \quad\text{ and }\quad
 \bb = (-1)^k\filler{k-1}\tp \lcu{x_kx_{k+1}}\tp x_kx_{k+1} \tp \filler{n-k}
\]
respectively. Since $k+1$ is initial, $[x_k]\not\preceq[x_{k+1}]$ and so
$\high_{L(\bb)}([b_k]) < \high_{L(\et)}([x_k])$; it follows that $\high(\bb)< \high(\et)$, since $\bb$ agrees with $\et$ in all other entries. A similar argument shows that $\high(\aa)<\high(\et)$.

In the case where $k=0$ (and $1\in I_\et$) then put $\bb = s_n\face_0(\et)$, $\bb'=\cyc^{-1}(\bb)$ (see \eqref{eq:dfn-cyc-shift}), and $\aa=\face_0s_1(\et)$. Then
\[
 \aa = \face_0s_1(\et) = - x_1\tp\filler{n-1}\tp \lcu{x_{n+1}}x_1
 \quad\text{ and }\quad
 \bb' = x_{n+1}x_1\tp \filler{n-1}\tp\lcu{x_{n+1}x_1} \,.
\]
The same arguments as in the first part of the proof show that $\high(\bb')$ and $\high(\aa)$ are both strictly less than $\high(\et)$; it remains only to note that since the height of an elementary tensor is unchanged by cyclic shifts, $\high(\bb)=\high(\bb')$.
\end{proof}

\begin{lem}\label{l:Noble}
Let $1\leq i \leq n$ and suppose $(i,i)\in \SD_j$. Then either $i$ or $i+1$ lies in $I_\et$. If $(n,0)\in\SD_j$ then either $n+1$ or $1$ lies in $I_\et$.

Consequently, if $\et$ has height $h$, then
\begin{equation}
\label{eq:sleep}
R_{\SD}(\et)
 + \sum_{1\leq i\leq n \st (i,i)\in\SD_j} s_i\face_i(\et)
\equiv \sum_{i\in I_\et} \filler{i-1}\tp \lcu{x_ix_{i+1}} \tp x_ix_{i+1} \tp \filler{n-i}
 \quad\mod G_{n,j,h-1}\,,
\end{equation}
where if $n+1\in I_\et$ the corresponding term on the right-hand side of~\eqref{eq:sleep} is interpreted as $(-1)^n\subchain{2}{n}\tp \lcu{x_{n+1}x_1}\tp x_{n+1}x_1$\/.
\end{lem}

\begin{proof}
Let $1\leq i\leq n$. Write $s_i\face_i(\et)= \aa$, as defined in the proof of Lemma~\ref{l:Donna}. By assumption, $\aa$ has $j$ left-blocks and one of them starts in position $i$. If neither $i$ nor $i+1$ were initial in $\et$, then $x_i$ and $x_{i+1}$ would both lie in the same left-block of $\et$, whose initial point is some $k<i$; and so $a_k$ would also mark the start of a left-block in $\aa$ which contains $a_i=\lcu{x_ix_{i+1}}$. Since we originally assumed that $i\in I_\aa$, this yields a contradiction.

A similar argument, with slight adjustments to the notation, shows that if $n\in I_{\face_0(\et)}$ then either $n+1$ or $1$ must have been initial in $\et$.
This completes the proof of the first part of the lemma.

For the second part of the lemma, suppose that $i\in I_\et$ with $1\leq i \leq n$, and note that there are two possibilities. Either $\face_i(\et)$ has fewer than $j$ left-blocks, in which case $s_i\face_i(\et)\in\cF_n^{j-1}$\/; or else $\face_i(\et)$ has exactly $j$ left-blocks, in which case one of them must start in position~$i$, and so $(i,i)\in\SD_j$.
By the first part of the lemma, the only other $(k,k)\in\SD_j$ with $1\leq k \leq n$ must arise from having $k+1\in I_\et$; but then by Lemma~\ref{l:Donna}\ref{item:sd-lower}, such terms have height at most $h-1$.

It remains to deal with the case where $n+1\in I_\et$. If $\face_0(\et)$ has fewer than $j$ left-blocks then $R_{\SD}(\et)=0$; if it has exactly $j$ left-blocks, then $(n,0)\in\SD_j$ and so
\[ R_{\SD}(\et) = (-1)^n \subchain{2}{n-1}\tp \lcu{x_{n+1}x_1}\tp x_{n+1}x_1\,,\]
as required.
Equation~\eqref{eq:sleep} now follows.
\end{proof}

Summing up: all terms in \eqref{eq:ds1} have strictly lower height than~$\et$; the terms in \eqref{eq:ds2} each give $\et$, since $\lcu{x_i}x_i=x_i$ for all $i$; and Lemma~\ref{l:Noble} tells us the sum of terms in \eqref{eq:sd} with $R_{\SD}(\et)$, provided we work modulo terms of fewer left-blocks or lower height. Therefore, the right-hand side of \eqref{eq:interim} is equal to
\[ \sum_{k\in I_\et} \et + \sum_{k\in I_\et} (-1)^{k+1}\face_{k+1}s_k(\et) + \sum_{k\in I_\et} (-1)^k s_k\face_k(\et)
\quad\mod G_{n,j,h-1}, \]
provided that we interpret the case $n+1\in I_\et$ appropriately. Expanding this out gives exactly what is claimed in Proposition~\ref{p:modified-main}, and so completes the proof.\hfill$\Box$
\end{section}

\begin{section}{Tying things together}
The inductive calculations done in the previous sections give us the following result.

\begin{thm}\label{t:maintech}
Let $n\geq 1$, and let $\psi\in\ZC^n(\ell^1(S))$ be an $R$-normalized cyclic $n$-cocycle. Then $\psi$ is a cyclic coboundary.
\end{thm}

Combining this with Proposition~\ref{p:initialization}, we finally obtain our main result.
\begin{thm}\label{t:headline}
The cyclic cohomology of $\ell^1(S)$ is zero in all odd degrees; whereas in even degrees, it is the space $\{ [\tau^{(2n)}] \;:\; \tau\in \SZ{0}{\ell^1(S)}\}$.
\end{thm}

As promised earlier, we can use Theorem~\ref{t:headline} to determine the simplicial cohomology of $\ell^1(S)$, via the Connes-Tzygan long exact sequence.
This requires one last fact about how the cohomology classes $[\tau^{(2n)}]$ transform under the shift map~$S$.

\begin{lem}\label{l:shift-of-traces}
Let $n\geq 1$\/. There exists a non-zero constant $\lambda_n$ such that,
for any Banach algebra~$A$ and $\tau\in \SZ{0}{A}$, the shift map $S:\HC^{2n-2}(A)\to \HC^{2n}(A)$ satisfies $S(\tau^{(2n-2)})= \lambda_n\tau^{(2n)}$.
\end{lem}

This lemma seems to be folklore, to an extent: for a direct proof that does not rely on \cite{He1}, see~\cite{YC_CTconst}. The value of $\lambda_n$ depends on a choice of scalar normalization of $S$ when one constructs the Connes-Tzygan exact sequence. In \cite{YC_CTconst} the formulas are chosen so that $\lambda_n=1$ for all $n$\/; but if one uses the formulas of~\cite{He1}, then different scaling factors will appear.

\begin{thm}\label{t:headline2}
The simplicial cohomology of $\ell^1(S)$ is zero in degrees $1$ and above.
\end{thm}

\begin{proof}
By Theorem~\ref{t:we-have-hunital} and \cite[Theorem~11]{He1}, the Connes-Tzygan sequence for $\ell^1(S)$ exists.
Then, since the cyclic cohomology of $\ell^1(S)$ vanishes in all odd degrees, the long exact sequence breaks up to give exact sequences
\[ 0\to \HH^{2n-1}(\ell^1(S)) \xrightarrow{B} \HC^{2n-2}(\ell^1(S)) \xrightarrow{S} \HC^{2n}(\ell^1(S)) \to \HH^{2n}(\ell^1(S)) \to 0 \]
for all $n\geq 1$\/.
Moreover, the shift map is surjective: for by Theorem~\ref{t:headline}, every cyclic $2n$-cocycle is cohomologous to one of the form $\tau^{(2n)}$ for some trace $\tau$, and by Lemma~\ref{l:shift-of-traces} we have $\tau^{(2n)}=\lambda_n^{-1}S\tau^{(2n-2)}$\/. Thus $\HH^{2n}(\ell^1(S))=0$ for all $n\geq 1$\/.

To finish, it suffices to show that the shift map is injective (which will imply that $B$ is the zero map, and hence that $\HH^{2n-1}(\ell^1(S))=0$\/).
As already observed in this proof, $\HC^{2n-2}$ is generated by cohomology classes of the form $[\tau^{(2n-2)}]$ where $\tau\in\SZ{0}{\cA}$. Consider $S([\tau^{(2n-2)}])=\lambda_n [\tau^{(2n)}]$ and suppose that $\tau^{(2n)}=\dif\varphi$ for some $\varphi\in\CC^{2n-1}(A)$\/. For each idempotent $e\in \cA$\/, direct calculation gives
$\tau(e)=\dif\varphi(e,\dots, e)(e) = \varphi(e,\ldots, e)(e)$.
But since $\varphi$ is \emph{cyclic}, $\varphi(e,\ldots,e)(e)= - \varphi(e,\ldots,e)(e)=0$\/. Thus $\tau$ vanishes on each idempotent in $\cA$, and since $\cA=\ell^1(S)$ where $S$ is a band, continuity forces $\tau$ to vanish identically.
\end{proof}

\end{section}

\begin{section}{Conclusion}

We have had to work quite hard to establish that the cyclic and simplicial cohomology of a band $\ell^1$-semigroup algebra behave as one would hope. Our methods would simplify in the case where the band is a semilattice, and in that case they would give an alternative approach to the main result of~\cite{Ch1}. Note that in~\cite{Ch1}, the author was unable to obtain explicit formulas for cobounding a given cocycle in high degrees, since the contracting homotopy in that setting was only given recursively. Here, we have an explicit algorithm for cobounding a given \emph{cyclic} cocycle; but once again we do not have a reasonable formula for cobounding arbitrary cocycles in high degrees, even for the case of a semilattice, owing to the reliance on the Connes-Tzygan exact sequence.

We feel that the tactics used in establishing the main result may be of wider interest, when interpreted in a broader sense. A general picture seems to be emerging: in order to obtain vanishing results for simplicial cohomology of Banach algebras, unless the geometry of the underlying Banach spaces intervenes helpfully, one has to replace the exhaustion arguments that are commonly found in `purely algebraic' cohomology of algebras, with more careful approaches: and these new arguments seem to depend on the \emph{local relations} between entries of a given elementary tensor in the Hochschild chain complex, rather than on how such entries factorize in terms of \emph{global generators} for the algebra.

\subsection*{Acknowledgments}
The first and second authors thank the School of Mathematics and Statistics at Newcastle University,
for its hospitality during several visits while this paper was being written. The work of the second author was supported in parts by a research grant from the Canadian NSERC.
\end{section}


\vfill

\begin{tabular}{l@{\hspace{12mm}}l}
Y. Choi and F. Gourdeau & M. C. White \\
D\'epartement de math\'ematiques  &
        School of Mathematics \\
\text{\hspace{1.0em}} et de statistique, &
\text{\hspace{1.0em}} and Statistics, \\
Pavillon Alexandre-Vachon &
        Herschel Building  \\
Universit\'e Laval &
          Newcastle University \\
Qu\'ebec, QC &
         Newcastle upon Tyne, Tyne \& Wear \\
Canada, G1V 0A6 &
         United Kingdom, NE1 7RU \\
        & \\
{\bf Email: \tt y.choi.97@cantab.net} &
        {\bf Email: \tt Michael.White@ncl.ac.uk} \\
{\bf Email: \tt Frederic.Gourdeau@mat.ulaval.ca} &
\end{tabular}

\end{document}